\documentclass[11pt]{amsart}

\usepackage{graphicx}
\usepackage[margin=1.1in]{geometry}
\usepackage{amsmath}
\usepackage{amssymb}
\usepackage{amsthm}
\usepackage{amsfonts}
\usepackage[export]{adjustbox}
\usepackage{graphicx}
\usepackage{float}
\usepackage{algorithm}
\usepackage{algpseudocode}
\usepackage[algo2e]{algorithm2e}
\usepackage{enumitem}
\usepackage{subfig}
\usepackage[cmyk,usenames,dvipsnames]{xcolor}
\usepackage{mathtools}
\usepackage{microtype}
\usepackage{times}

\usepackage{hyperref}
\usepackage{pdfsync}
\usepackage{dsfont}
\usepackage{color}

\usepackage{cite}
\usepackage{bbm}
\usepackage{bm}
\usepackage{multirow}
\usepackage[font=footnotesize,labelfont=bf]{caption}
\usepackage{bibentry}
\nobibliography*

\setlist[1]{labelindent=\parindent}
\setlist[enumerate,1]{label = \Roman*., ref = \Roman*}
\setlist[enumerate,2]{label = \Alph*., ref = \Alph*}
\setlist[enumerate,3]{label = \roman*., ref = \roman*}
\setlist[enumerate,4]{label = \alph*., ref = \alph*}

\graphicspath{{./figures/}}

\numberwithin{equation}{section}

\newtheorem{theorem}{Theorem}[section]

\newtheorem{lemma}[theorem]{Lemma}
\newtheorem{corollary}[theorem]{Corollary}

\newtheorem{problem}[theorem]{Problem}

\theoremstyle{remark}
\newtheorem{remark}[theorem]{Remark}
\newtheorem{assumption}[theorem]{Assumption}

\newcommand{\pri}{\mathrm{prior}}

\newcommand{\EE}{\mathbb{E}}

\newcommand{\Cov}{\mathrm{Cov}}

\newcommand{\MCG}{\mathcal{G}}

\newcommand{\ini}{\mathrm{ini}}

\newcommand{\rd}{\mathrm{d}}

\newcommand{\KL}{\textsf{KL}}
\newcommand{\Tr}{\textrm{Tr}}

\newcommand{\Rb}{\mathbb{R}}

\newcommand{\Eb}{\mathbb{E}}

\newcommand{\Sb}{\mathbb{S}}

\newcommand{\Cc}{\mathcal{C}}
\newcommand{\Dc}{\mathcal{D}}

\newcommand{\Kc}{\mathcal{K}}
\newcommand{\Lc}{\mathcal{L}}

\newcommand{\Pc}{\mathcal{P}}
\newcommand{\Qc}{\mathcal{Q}}
\newcommand{\Oc}{\mathcal{O}}

\newcommand{\rmc}{\mathrm{c}}
\newcommand{\rmd}{\mathrm{d}}
\newcommand{\rmb}{\mathrm{b}}

\newcommand{\rmB}{\mathrm{B}}

\newcommand{\rmN}{\mathrm{N}}

\newcommand{\Law}{\mathrm{Law}}

\newcommand{\rhotarg}{\rho_{\text{target}}}

\DeclareMathOperator*{\argmin}{\mathrm{argmin}}

\title[Many particle Bayesian sampling]{Bayesian sampling using interacting particles}

\author{Shi Chen}
\address{Mathematics Department, University of Wisconsin-Madison, 480 Lincoln Dr., Madison, WI 53705 USA.}
\email{schen636@wisc.edu}

\author{Zhiyan Ding}
\address{Department of Mathematics, University of California, Berkeley, CA 94720, USA.}
\email{zding.m@math.berkeley.edu}

\author{Qin Li}
\address{Mathematics Department, University of Wisconsin-Madison, 480 Lincoln Dr., Madison, WI 53705 USA.}
\email{qinli@math.wisc.edu}

\date{\today}

\begin{document}

\begin{abstract}
Bayesian sampling is an important task in statistics and machine learning. Over the past decades, many ensemble-type sampling methods have been proposed. In contrast to the classical Markov chain Monte Carlo methods, these new methods deploy a large number of interactive samples, and the communication between these samples is crucial in speeding up the convergence. To justify the validity of these sampling strategies, the concept of interacting particles naturally calls for the mean-field theory. The theory establishes a correspondence between particle interactions encoded in a set of coupled ODEs/SDEs and a PDE that characterizes the evolution of the underlying distribution. This bridges numerical algorithms with the PDE theory used to show convergence in time. Many mathematical machineries are developed to provide the mean-field analysis, and we showcase two such examples: the coupling method and the compactness argument built upon the martingale strategy. The former has been deployed to show the convergence of the ensemble Kalman sampler and the ensemble Kalman inversion, and the latter will be shown to be immensely powerful in proving the validity of the Vlasov-Boltzmann simulator.
\end{abstract}

\maketitle

\setcounter{tocdepth}{2}
\tableofcontents


\section{Introduction}\label{sec:intro}
Sampling is a fundamental task in the field of machine learning and serves as a crucial building block in various applications~\cite{MCMCforML}. For example, sampling plays a significant role in Bayesian statistics, data assimilation~\cite{Reich2011}, generative modeling~\cite{SoSoKiKuErPo:2020score}, volume computation~\cite{Convexproblem}, and bandit optimization~\cite{pmlr-v119-mazumdar20a,ATTS,10.5555/2503308.2343711}. Sampling methods have been widely used in various scientific domains, including atmospheric science, petroleum engineering, and epidemiology~\cite{FABIAN198117,PES,COVID_travel}.

The goal of sampling is to obtain approximately independent, identically distributed (i.i.d.) samples that come from a target distribution. Given a target distribution with density
\begin{equation}\label{def:target}
\rhotarg(x) \propto e^{-f(x)}\,,\quad \text{with }\quad f(x)\,:\mathbb{R}^d\to\mathbb{R}\,,
\end{equation}
we are to find
\[
\{x_i\}_{i=1}^N\sim \rhotarg\,.
\]
It is typically assumed $d\gg 1$. A classical origin of the sampling problem is Bayesian inference, in which the function $f$ takes the form of
\begin{equation}\label{eqn:cost_bayes}
f(x;y) = \Phi(x;y)+\frac{1}{2}\left|x-x_0\right|^2_{\Gamma_0}\,,\quad\text{with}\quad\Phi(x;y)=\frac{1}{2}\left|y-\mathcal{G}(x)\right|^2_\Gamma\,,
\end{equation}
where $\mathcal{G}:\mathbb{R}^d\rightarrow\mathbb{R}^{d'}$ denotes the forward map that maps input $x$ to the output $\mathcal{G}(x)$ that hopefully agrees with the data $y\in\mathbb{R}^{d'}$, and $\Phi$ is to penalize the measuring error. $\Gamma$ and $\Gamma_0\in\mathbb{R}^d$ are positive semidefinite matrices that serve as the covariance matrix in the domain and range of the map with $|\ \cdot\ |_{\Gamma_0}:=|\Gamma_0^{-\frac{1}{2}}\ \cdot\ |\,$ and $|\ \cdot\ |_\Gamma:=|\Gamma^{-\frac{1}{2}}\ \cdot\ |\,$.

This formulation stems from the Bayesian theorem that states
\begin{equation}\label{eqn:bayes}
\rhotarg(x) =\rho_{\text{post}}(x)\propto \exp(-f(x))= \rho_{\text{prior}}(x)l(y;x)
\end{equation}
with $\rho_\text{prior}(x)\propto \exp\left(-\frac{1}{2}\left|x-x_0\right|^2_{\Gamma_0}\right)$, and the likelihood function $l(y;x) =\exp\left(-\Phi(x;y)\right)$ that depends on the forward map $\mathcal{G}$. Since $y$ is the given data and is not subject to change, when the context is clear, we drop its dependence in $l$, $f$, and $\Phi$.

Sampling from a target distribution has a rich literature. Dating back to 1953 with the development of the Metropolis algorithm, for several decades, Markov chain Monte Carlo (MCMC) algorithms have been the primary tool for generating samples~\cite{GiRiSp:1995markov,GaLo:2006markov,RoCaCa:1999monte}. Different MCMC methods adopt different dynamics, but they share the same mathematical structure. Each MCMC method is run according to a transition matrix, and the target distribution is the unique invariant measure of such transition matrix. Subsequently, as a sample evolves according to this transition matrix, and its final distribution approaches the invariant measure, meaning that the sample can be viewed as a drawing from the target distribution. Over the decades, MCMC methods have attracted many investigations, and we have obtained a collection of results on MCMC convergence, including the convergence rate on the continuous and discrete level, both asymptotically and non-asymptotically with explicit rates~\cite{durmus2017,doi:10.1111/rssb.12183, DALALYAN20195278, durmus2018analysis, Cheng2017UnderdampedLM,eberle2019,cao2019explicit,pmlr-v75-dwivedi18a,wu2022minimax,Chewi2020OptimalDD,10.5555/3495724.3497366,pmlr-v178-chewi22a,DBLP:conf/nips/ShenL19,DBLP:conf/colt/LeeST20,NEURIPS2019_65a99bb7}.

In contrast to the classical single-particle MCMC methods, a new line of techniques surrounding ``ensemble method'' emerged over the past decade and have garnered a lot of interest. Instead of moving a single particle, many particles evolve simultaneously and interactively. The concept of utilizing a many-particle system instead of one particle traces back to the Ensemble Kalman Filter~\cite{Evensen_enkf,Ghil1981,Evensen2003,firstEnkf}, an ensemble-based version of the Kalman filter~\cite{K_1960,KB_1961} that stores Kalman matrices through sample representations. In the early 2010s, various ensemble-based methods were developed, including Affine Invariant Ensemble Sampler~\cite{goodman2010ensemble}, Ensemble Kalman Inversion~\cite{DAEnKF,Iglesias_2013} and, more recently, Consensus Based Sampling~\cite{CaHoStVa:2022consensus,gerber2023meanfield}, Stein Variational Gradient Descent~\cite{LiWa:2016stein,LuLuNo:2019scaling}, Langevin Sampling with Birth-death~\cite{LuLuNo:2019accelerating}, and Ensemble Kalman Sampler~\cite{EKS}, and their derivatives such as the Weighted Ensemble Kalman Inversion~\cite{ding2020ensemblecorrect}.

It is believed that ensemble methods have some advantages over single-particle MCMC methods. The argument is that if many particles are exploring the full landscape together, the chance of them covering the full domain and discovering all local and global minima is higher, and thus can potentially capture the target distribution faster. Yet, this promise has not been fully realized. The first step towards this is to spell out explicitly the convergence rate of the ensemble: With what convergence rate in time and in the number of samples can the ensemble converge to the target distribution?

Unfortunately, unlike single-particle MCMC solvers, the numerical analysis for ensemble-type algorithms is considerably more intricate. Many particles interact in a nonlinear manner. It is not immediate that such an interaction speeds up the convergence. Unlike MCMC, which typically requires one to trace the evolution of the distribution of the single sample through the application of the transition matrix, for an ensemble, since many samples are involved, the underlying distribution lives on a tensorized space that is significantly higher-dimensional. This perspective necessitates the deployment of the mean-field limit to compress the tensorized probability space down to the original space. Upon such compression, the many-particle system is altogether described by an ensemble (empirical) distribution, providing a fair platform to compare with the classical MCMC evolution.

A variety of approaches can be used to construct a sampling algorithm involving many particles, but the general procedure is usually the same. Generally speaking, there are three steps:
\begin{itemize}
    \item[1.] Identify an evolutional PDE whose long time behavior recovers an invariant measure that coincides with the target distribution;
    \item[2.] Design its particle method and translate the evolutional PDE for the underlying distribution to a coupled ODE/SDE system for each particle;
    \item[3.] Discretize the ODE/SDE system by implementing the Euler-Maruyama or high-order discretization methods.
\end{itemize}
Accordingly, the analysis of the numerical error is divided into three sections.
\begin{itemize}
    \item Show that the solution to the PDE indeed converges to the target distribution efficiently. This is mostly PDE analysis.
    \item Prove that the PDE can be solved by the particle method. This is to show that the ensemble distribution of the $N$-particle accurately recovers the PDE solution. This step equates the PDE and the coupled ODE/SDE system, and is typically called deriving the mean-field limit.
    \item Demonstrate that the SDE discretization accurately recovers the dynamics of the underlying SDE system. This corresponds to analyzing the discretization error and showing that it is small when the stepsize $h$ is small.
\end{itemize}
We summarize the above three steps in Figure~\ref{fig:three_stage}, and the associated strategy deployed to validate the algorithms. We should stress that these steps are ordered, and the order cannot be switched. Mathematically analyzing the algorithm amounts to taking the following limits:
\[
\lim_{t\to\infty}\lim_{N\to\infty}\lim_{h\to0}\,.
\]
The three limits do not commute, and the analysis has to be done accordingly. In particular, for a preset accuracy level $\epsilon$, running the analysis in Step 1 amounts to pin a time $t$ so that the PDE solution at this time is within $\frac{\epsilon}{3}$ to the target distribution; and for this $t$, we run analysis in Step 2 to identify the $N$, to ensure the $N$ coupled SDE solution approximates the PDE within $\frac{\epsilon}{3}$; and finally, for the picked $(t,N)$, we run analysis in Step 3 to identify a small enough $h$ that ensures the discrete solution is within $\frac{\epsilon}{3}$ to the continuous-in-time SDE solution. This ordered constraint comes from the techniques that we deploy, such as the Gr\"onwall inequality. It is possible for the limits to commute, if other machinery gets deployed. We do not discuss these options in this paper.

Each of these three steps of justification carries its own challenges and attracts a great deal of investigation. In this paper, we only discuss the mean-field analysis component of it (Step 2), using the Ensemble Kalman Sampler, Ensemble Kalman Inversion, and Boltzmann simulator as examples. The other two steps are equally important in the whole algorithm justification process, but are not the main focus of the paper. We refer the interested readers to~\cite{DeVi:2005trend,Ki:2014boltzmann,ding2020ensemblecorrect,Carrillo_2021,carrillo2023mean}.

To a large extent, the goal of proving ``mean-field limit'' and justifying particle methods for simulating PDEs very much align. Essentially one claims that a distribution density function can be represented as an ensemble of many particles drawn from it. Accordingly, the evolution of this density function encoded in a dynamical PDE can be translated to the evolution of these representative particles. The question then becomes quantifying the differences between these two systems when the number of representative particles is large.

Multiple approaches have been taken to achieve mean field derivation. The most straightforward machinery is the coupling method~\cite{Sznitman,Fournier2015,Carrilo2011,Blob,Bolley_Carrillo}. This amounts to designing an auxiliary particle system with all particles passively pushed by the underlying ``field.'' This auxiliary system serves as a bridge to link the underlying distribution and the ensemble system. When the system demonstrates good properties, such as Lipschitz forces, the method becomes very handy, as will be summarized in Section~\ref{sec:unif_analysis}. Often in practice, forces, though regular enough, are not globally Lipschitz, preventing the direct application of the coupling method. Depending on the context, a mathematical treatment specifically designed for the SDEs at hand is required. This is a case-by-case discussion, and we will dive in detail in Section~\ref{sec:ensemble}, using the ensemble Kalman inversion (EKI)\cite{Iglesias_2013} and the ensemble Kalman sampler (EKS)\cite{EKS} as examples.

Another line of mean-field analysis is termed the martingale method. Within this framework, one can translate the PDE problems into their associated martingale problems that look for solutions as probability measures over the space of path measures. In this path measure space, one first demonstrates the tightness for a family of path measures. The tightness allows one to pass to a limit, upon which one can show that this limit happens to solve the mean-field PDE. Since the martingale method translates the mean-field convergence to the weak convergence of a sequence of measures over stochastic processes,
the notion of convergence is quite weaker. Weak convergence is deployed to accommodate systems that are not so smooth, and thus the approach can handle a much larger class of interacting particle systems. In particular, we will use it to show the Vlasov-Boltzmann simulator in Section~\ref{sec:Boltzmann}. The Boltzmann collision operator is significantly more complex, and mean-field analysis is beyond what the coupling method can handle. However, while the method is much more versatile and relaxes the assumptions on systems' coefficients, it lacks a precise description of the convergence rate. In Section~\ref{sec:boltzmann_numerics}, we present some numerical evidence.

The mean-field limit results that deploy the coupling method (Section~\ref{sec:ensemble}) have appeared in literature, and we specifically mention~\cite{ding2019ensemble} and~\cite{Zhi_Qin_2021}. The machinery is rather standard and the set of theory is complete. 
Section~\ref{sec:Boltzmann} discusses the Boltzmann simulator to perform Bayesian sampling. For this new design of sampling method, we adopt the martingale strategy utilized in showing the mean-field limit for the classical Boltzmann equation, and transition it into our new artificial system.

We finally stress that Bayesian sampling and optimization are two machine learning domains that are closely tied to each other, and they share mathematical strategies. Mean-field limit and interactive particles are also heavily relied on in the design and analysis of optimization solvers, typically under the named of ``consensus-based" or ``swarm-based." This is a vast domain of research and we can only list a few recent studies~\cite{Pa:2024optimization,BeBoPa:2022binary,BoPa:2023kinetic,carrillo2024interacting,gerber2023meanfield,CaChToTs:2018analytical,LuTaZe:2022swarm,TaZe:2023swarm,ding2024swarmbased}.

\begin{figure}
    \centering
    \includegraphics[width=0.8\textwidth]{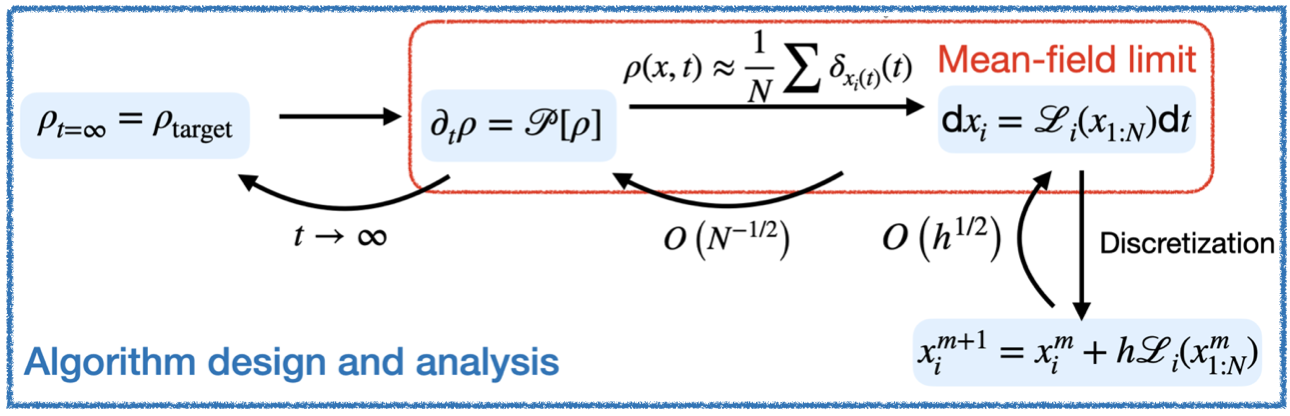}
    \caption{The process of algorithm design (left to right) typically involves three steps: designing a PDE, designing its particle method, discretization. The analysis procedure (right to left) is consistent with the design process in reverse. This paper only focuses on the second step of algorithm analysis, which involves applying mean-field analysis to validate the mean-field limit of multiple-particle algorithms.}
    \label{fig:three_stage}
\end{figure}

\section{Mean-field limit for ensemble based sampling methods}\label{sec:ensemble}
Ensemble based sampling methods are immensely popular in areas related to atmospheric science and weather prediction. The concept originates from Ensemble Kalman Filter, an algorithmic pipeline for implementing Kalman Filter. Kalman Filter, as a standard technique for data assimilation, allows one to gradually assimilate information from collected data to dynamical models that have unknown parameters, so to gradually refine our estimates for these parameters. In the most brute-force simulation, the Kalman filter calls for the computation of the Kalman gain matrix, a term strongly tied to the covariance of the parameters. This is a very expensive numerical step when the problem is high-dimensional. Ensemble Kalman Filter replaces the computation of the covariance by the ensemble covariance that can be extracted from a large quantity of samples. This way, one avoids direct computing the covariance and it significantly reduces the computational cost~\cite{Evensen_enkf,Ghil1981,Evensen2003,firstEnkf,DAEnKF,LeGland,law_tembine_tempone}.

The idea of using the ensemble covariance to replace the true covariance was later adopted by the field of Bayesian sampling, see~\cite{calvello2022ensemble,Reich2011} and the references therein, e.g.~\cite{Bergemann_Reich10_local,Bergemann_Reich10_mollifier}. Starting in the early 2000s, the idea is at the core of a long list of exciting algorithm developments~\cite{Iglesias_2013,EKS,ding2020ensemblecorrect}. A similar perspective-shifting took place in optimization as well. Ensemble based, or swarm based methods were designed to specifically tackle optimization problems that are non-convex, and has since generated tremendous excitement~\cite{CaChToTs:2018analytical,LuTaZe:2022swarm,TaZe:2023swarm}. 

Two distinctive and representative solvers in this research direction are the Ensemble Kalman Inversion (EKI)~\cite{Iglesias_2013} and the Ensemble Kalman Sampler (EKS)~\cite{EKS}. The two methods are widely used to predict geophysical turbulence~\cite{DuIaXi:2019turbulence,GuHaScHuDuLoWu:2023posteriori,ScLaStTe:2017earth}, modeling reservoirs~\cite{IgLaSt:2013evaluation}, and modeling fluid flow from the cardiovascular system~\cite{QuMaVe:2017cardiovascular}, and generated excitement in various scientific domains.

We are interested in understanding the validity of the algorithms. As discussed in Section~\ref{sec:intro}, we are particularly interested in showing the ``mean-field'' component of the method justification. Both systems deploy covariance matrices for samples to communicate with each other, and the flow fields for the samples are rather smooth. These properties allow us to implement the coupling method to conduct the mean-field analysis~\cite{ding2019ensemble,Zhi_Qin_2021} and demonstrate the convergence rate of the mean-field limit.

Below, we will first present in Section~\ref{sec:unif_analysis} a unified framework that discusses a typical use of the coupling method. Sections \ref{sec:EKI} and~\ref{sec:EKS} are designated for EKI and EKS, respectively.

\subsection{Unified coupling method framework}\label{sec:unif_analysis}
\begin{figure}[htb]
    \centering
    \includegraphics[width=0.8\textwidth]{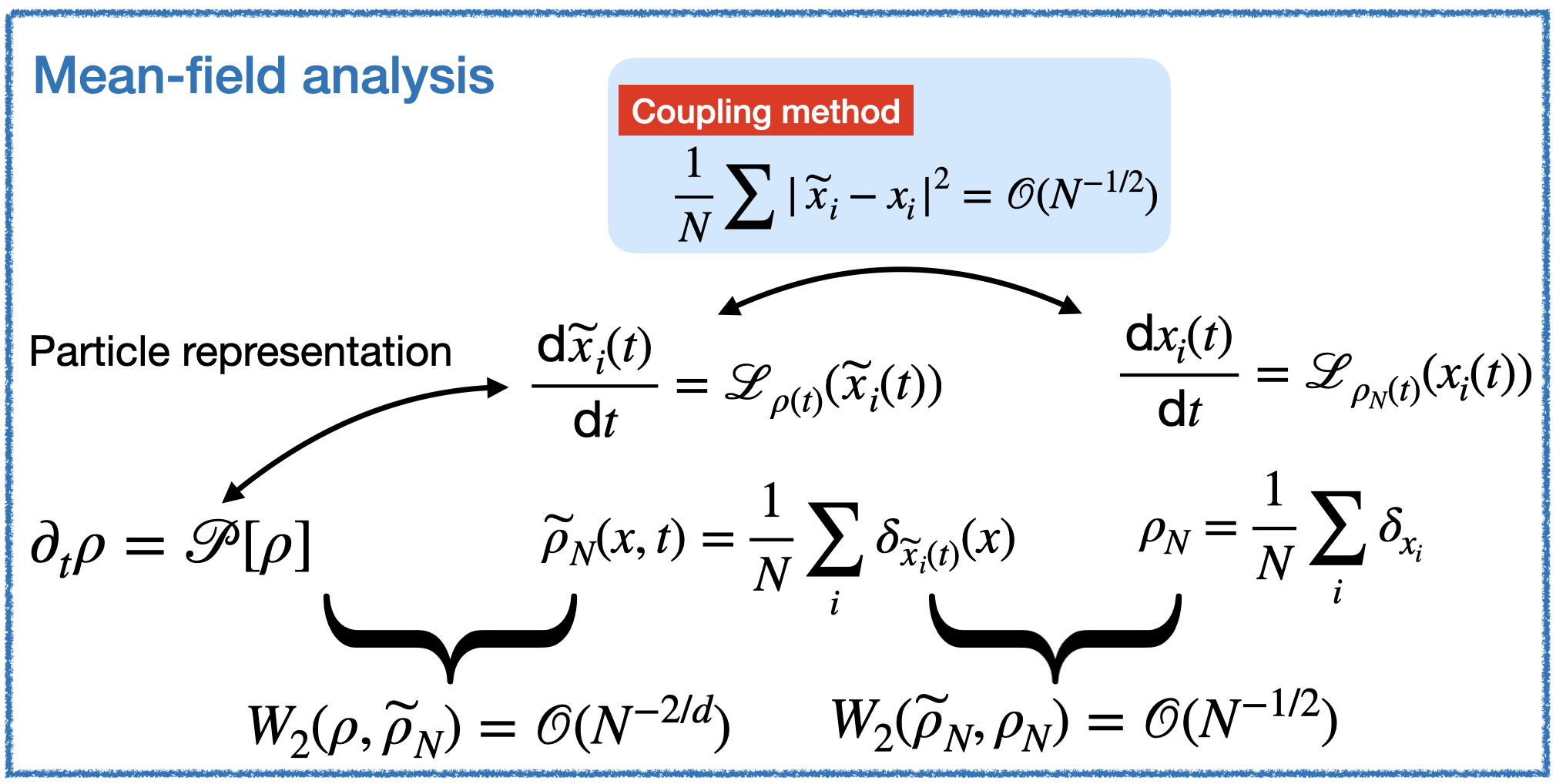}
    \caption{The general routine of the mean-field analysis from coupling perspective. The auxiliary particles $\widetilde{x}_i(t)$ evolve according to a coupled ODE/SDE with $\rho(t)$ generated from the mean-field limiting PDE. The coupling method shows that $x_i$ is close to $\widetilde{x}_i$, which implies that $\widetilde{\rho}_N$ is close to $\rho_N$. This, combined with the fact that $\widetilde{\rho}_N\approx \rho$ from \cite{Fournier2015}, ultimately justifies $\rho_N\approx\rho$.}
    \label{fig:coupling}
\end{figure}

The coupling method was beautifully shown in~\cite{Sznitman,Vi:2009optimal} and this framework can handle a large class of systems that have nice properties. We take Fokker-Planck type PDE for the illustration:
\begin{equation}\label{eqn:rho_pde}
\begin{aligned}
    \partial_t\rho(t,x) &=\mathcal{P}[\rho(t,x))]\\
    &= \nabla_x\cdot(F(t,x;\rho)\rho(t,x))) +\frac{1}{2}\nabla_x\cdot(D(t,x;\rho)^\top D(t,x;\rho)\nabla_x\rho(t,x))\,,
\end{aligned}
\end{equation}
where $\rho(t)$ is the PDE solution. The operator $\mathcal{P}$ is composed of two components, the drift term with the coefficient $F(t,x;\rho)$ and the diffusion term with the coefficient $D(t,x;\rho)$. Both $F$ and $D$ can depend on $\rho$, and when they do, this equation is nonlinear. Throughout the section, we term the equation System A.

This system is closely related to the following two types of SDEs.
\begin{itemize}
\item[System B:] The particle system pushed by the underlying characteristics. Denote $\{\tilde{x}_i\}$ a collection of samples i.i.d. drawn from $\rho(t=0,\cdot)$ at the initial time $t=0$. These samples evolve according to:
\begin{equation}\label{eqn:auxillary}
\rd\tilde{x}_i(t)=F(t,\tilde{x}_i(t);\rho)\rd t+D(t,\tilde{x}_i(t);\rho)\rd W^i_t\,,
\end{equation}
where $W^i_t$ are independent Brownian motions. We further denote the ensemble distribution
\begin{equation}\label{eqn:rho_n_tilde}
\widetilde{\rho}_{N}(t,x):=\frac{1}{N}\sum^N_{n=1}\delta_{\widetilde{x}_i(t)}\,.
\end{equation}
The system is driven by the true underlying characteristics and thus honestly reflects the flow field. We expect $\widetilde{\rho}_{N}$ provided by System B to approximate $\rho$. However, we should note that~\eqref{eqn:auxillary} is not a system that can be numerically realized since its parameters depend on $\rho(t,\cdot)$, the solution to the PDE, and is unknown.
\item[System C:] The particle system driven by the self-generated field. Let $\{x_i(t)\}_{i=1}^N$ be a collection of $N$ samples i.i.d. drawn from $\rho(t=0,\cdot)$ at time $t=0$. We evolve them according to:
\begin{equation}\label{eqn:algorithm}
\rd{x}_i(t)=F(t,{x}_i(t);\rho_N)\rd t+D(t,x_i(t);\rho_N)\rd W^i_t\,,
\end{equation}
where drift and diffusion coefficient $F$ and $D$ depend on the self-induced ensemble distribution:
\begin{equation}\label{eqn:approx}
    \rho_N(t,x)=\frac{1}{N}\sum_i\delta_{x_i(t)}(x)\,.
\end{equation}
The system looks almost identical to System B in~\eqref{eqn:auxillary}. However, this new system deploys a self-generated $\rho_N$ instead of relying on an external $\rho$, and thus can be implemented numerically.
\end{itemize}

The mean-field derivation is to show that the solution of the underlying PDE~\eqref{eqn:rho_pde}, in System A, can be well approximated by the ensemble $\rho_N$ in~\eqref{eqn:approx} in System C. This way, one translates the computation of the underlying PDE to the particle simulation in~\eqref{eqn:algorithm}.

The coupling method amounts to showing the similarity between System A and System C by equating both with the middle agent: System B. Indeed, according to~\cite{Fournier2015}, System A and B are already close:
\begin{equation}\label{eqn:FG_convergence}
\text{dist}(\rho,\widetilde{\rho}_N)\to 0\,,\quad\text{as}\quad N\to\infty\,,
\end{equation}
for some proper definition of the $\text{dist}$ function\footnote{According to ~\cite{Fournier2015}, the $W_2$-distance between the measures $\rho\mathrm{d}x$ and $\widetilde{\rho}_N\mathrm{d}x$ converges to $0$ when $N$ approaches to infinity. We omit the symbol $\mathrm{d}x$ for simplicity of the notation when the context is clear.}. Most of the technical component of the proof goes to showing that System B and C are also close, meaning that, for all $0\leq t<\infty$,
\begin{equation}\label{eqn:coupling_error}
\frac{1}{N}\sum_j\mathbb{E}\left|\tilde{x}_j-x_j\right|^2\xrightarrow{N\to \infty}0\,,\quad\Rightarrow\quad \text{dist}(\widetilde{\rho}_N,{\rho}_N)\xrightarrow{N\to\infty}0\,.
\end{equation}

If this can be proved, one can run the triangle inequality for:
\[
\text{dist}(\rho,\rho_N)\leq \text{dist}(\rho,\tilde{\rho}_N)+\text{dist}(\widetilde{\rho}_N,{\rho}_N)\to 0\,,\quad\text{as}\quad N\to\infty\,,
\]
concluding the mean-field derivation.

To show~\eqref{eqn:coupling_error}, we note the similarity in form between~\eqref{eqn:auxillary} and~\eqref{eqn:algorithm}. The two sets of SDEs are presented in almost identical form, and mathematically it is straightforward to compute the difference. A typical strategy is to subtract the two systems and evolve their differences.
\begin{equation}\label{eqn:diff}
\rd\left(\tilde{x}_i(t)-x_i(t)\right) = \left[F(t,\tilde{x}_i(t);\rho)-F(t,x_i(t);\rho_N)\right]\rd t+\left[D(t,\tilde{x}_i(t);\rho)-D(t,x_i(t);\rho_N)\right]\rd W^i_t\,.
\end{equation}
Considering $x_i(t=0)=\tilde{x}_i(t=0)$, we need to show that the two terms on the right-hand side are controllable. This is indeed true when some proper assumptions are made. We now give a quick overview of the general recipe for doing so.

Deploying It\^o's formula, we obtain
\begin{equation}\label{eqn:W_2_diff}
\begin{aligned}
&\frac{\mathrm{d}}{\mathrm{d}t}\mathbb{E}\left(\left|\tilde{x}_i(t)-x_i(t)\right|^2\right)\\
=&2\underbrace{\mathbb{E}\left(\left(\tilde{x}_i(t)-x_i(t)\right)\left[F(t,\tilde{x}_i(t);\rho)-F(t,x_i(t);\rho_N)\right]\right)}_{\text{Term I}}\\
&+\underbrace{\mathbb{E}\left(\left\|D(t,\tilde{x}_i(t);\rho)-D(t,x_i(t);\rho_N)\right\|^2_F\right)}_{\text{Term II}}\,.    
\end{aligned}
\end{equation}
The cross terms vanish because $\mathbb{E}(\mathrm{d}W^i_t)=0$. The two terms are dealt with separately:
\begin{itemize}
\item To deal with Term I, we first notice
\begin{equation}\label{eqn:velocity_control}
\begin{aligned}
&\left|F(t,\tilde{x}_i;\rho)-F(t,x_i;\rho_N)\right|\\
\leq &\left|F(t,\tilde{x}_i;\rho)-F(t,\tilde{x}_i;\widetilde{\rho}_N)\right|+\left|F(t,\tilde{x}_i;\widetilde{\rho}_N)-F(t,x_i;\rho_N)\right|\\
\leq &\underbrace{\left|F(t,\tilde{x}_i;\rho)-F(t,\tilde{x}_i;\widetilde{\rho}_N)\right|}_{\leq LW_2(\rho,\widetilde{\rho}_N)} + \underbrace{\left|F(t,\tilde{x}_i;\widetilde{\rho}_N)-F(t,x_i;\widetilde{\rho}_N)\right|}_{\leq L|\tilde{x}_i-x_i|}+\underbrace{\left|F(t,{x}_i;\widetilde{\rho}_N)-F(t,x_i;\rho_N)\right|}_{\leq LW_2(\rho_N,\widetilde{\rho}_N)}
\end{aligned}
\end{equation}
In this inequality, we repeatedly used the triangle inequality, and the last line assumes the Lipschitz condition of $F$ on both $\rho$ and $x$, with $L$ denoting the Lipschitz constant over both arguments. Plugging \eqref{eqn:velocity_control} into Term I, we obtain
\begin{equation}\label{eqn:term_1}
\mathrm{Term\ I}\leq C\mathbb{E}\left(|\tilde{x}_i(t)-x_i(t)|^2+W^2_2(\rho_N,\widetilde{\rho}_N)+W^2_2(\rho,\widetilde{\rho}_N)\right)\,,    
\end{equation}
where $C$ is a constant depending on the Lipschitz constants, and Young's inequality is used.
\item To deal with Term II, we use similar strategy. Assuming $D$ is Lipschitz with respect to both $\rho$ and $x$ with constant $L$, we have
\begin{equation*}\label{eqn:diffusion_control}
\begin{aligned}
&\left\|D(t,\tilde{x}_i;\rho)-D(t,x_i;\rho_N)\right\|_F\\
\leq &\left\|D(t,\tilde{x}_i;\rho)-D(t,\tilde{x}_i;\widetilde{\rho}_N)\right\|_F+\left\|D(t,\tilde{x}_i;\widetilde{\rho}_N)-D(t,x_i;\rho_N)\right\|_F\\
\leq &\underbrace{\left\|D(t,\tilde{x}_i;\rho)-D(t,\tilde{x}_i;\widetilde{\rho}_N)\right\|_F}_{\leq LW_2(\rho,\widetilde{\rho}_N)} + \underbrace{\left\|D(t,\tilde{x}_i;\widetilde{\rho}_N)-D(t,x_i;\widetilde{\rho}_N)\right\|_F}_{\leq L|\tilde{x}_i-x_i|}+\underbrace{\left\|D(t,{x}_i;\widetilde{\rho}_N)-D(t,x_i;\rho_N)\right\|_F}_{\leq LW_2(\rho_N,\widetilde{\rho}_N)}\,,
\end{aligned}
\end{equation*}
which in turn bound Term II as
\begin{equation}\label{eqn:term_2}
\mathrm{Term\ II}\leq C\mathbb{E}\left(|\tilde{x}_i(t)-x_i(t)|^2+W^2_2(\rho_N,\widetilde{\rho}_N)+W^2_2(\rho,\widetilde{\rho}_N)\right)\,,
\end{equation}
where $C$ is a constant depending on the Lipschitz constants.
\end{itemize}
Plugging \eqref{eqn:term_1} and \eqref{eqn:term_2} in \eqref{eqn:W_2_diff} and averaging over $i$, we derive:
\[
\begin{aligned}
&\frac{\mathrm{d}}{\mathrm{d}t}\left(\frac{1}{N}\sum^N_{i=1}\mathbb{E}\left(\left|\tilde{x}_i(t)-x_i(t)\right|^2\right)\right)\\
\leq &C\left(\frac{1}{N}\sum^N_{i=1}\mathbb{E}\left(\left|\tilde{x}_i(t)-x_i(t)\right|^2\right)+\mathbb{E}\left(W^2_2(\rho_N,\widetilde{\rho}_N)\right)+\mathbb{E}\left(W^2_2(\rho,\widetilde{\rho}_N)\right)\right)\\
\leq &C'\left(\frac{1}{N}\sum^N_{i=1}\mathbb{E}\left(\left|\tilde{x}_i(t)-x_i(t)\right|^2\right)+\mathbb{E}\left(W^2_2(\rho,\widetilde{\rho}_N)\right)\right)\\
\leq & C'\left(\frac{1}{N}\sum^N_{i=1}\mathbb{E}\left(\left|\tilde{x}_i(t)-x_i(t)\right|^2\right)\right)+\mathcal{O}\left(N^{-\gamma}\right)\,.
\end{aligned}
\]
In the second inequality, we absorb $W^2_2(\rho_N,\widetilde{\rho}_N)$ because $W^2_2(\rho_N,\widetilde{\rho}_N) \leq \frac{1}{N}\sum^N_{i=1}\mathbb{E}\left(\left|\tilde{x}_i(t)-x_i(t)\right|^2\right)$, and in the last, we bounded $\mathbb{E}\left(W^2_2(\rho,\widetilde{\rho}_N)\right)$ with  $\gamma > 0$ being a constant that depends on the dimension. This estimate comes from an off-the-shelf result~\cite{Fournier2015}, and will be presented in Lemma \ref{lem:vj_FP}.

Finally, applying the Gr\"onwall inequality, we conclude that $\tilde{x}_i(t)$ and $x_i(t)$ are close for all finite times when $N$ is sufficiently large, thereby confirming~\eqref{eqn:coupling_error}.

The above intuitive derivation largely depends on the Lipschitz condition of the drift coefficients $F$ and diffusion coefficients $D$. Unfortunately, most ensemble methods only have locally Lipschitz coefficients, and for such methods, care has to be given in calculations in different contexts. In particular, in the mean-field proof for EKI, a boot-strapping argument is called, and for EKS, the Ando-Hemmen lemma is used. The choice of these arguments is not unique. A new technique was recently introduced in a beautiful set of work in~\cite{vaes2024sharp} that utilizes the concept of stopping-time to enforce a Lipschitz condition before the stopping time, and this condition is matched together with the application of the Burkholder–Davis–Gundy inequality to pass to use of the Gr\"onwall inequality. It seems that the framework is drastically more general than the bootstrapping technique, but more investigation is still needed to understand the mechanism. Nevertheless, the coupling method strategy remains the backbone of all these proofs.

Throughout the section, various statistical quantities appeared multiple times. We unify the notation here. For a given probability density $\rho(t,x)$, we define expectations
\[
\mathbb{E}_{\rho_t}(x)=\int_{\mathbb{R}^d}x\rho(t,x)\rd x,\quad \mathbb{E}_{\rho_t}(\mathcal{G}(x))=\int_{\mathbb{R}^d}\mathcal{G}(x)\rho(t,x)\rd x\,,
\]
and the covariance matrices
\begin{equation}\label{eqn:cov_rho}
\Cov_{\rho_t} = \int_{\mathbb{R}^d}\left(x-\mathbb{E}_{\rho_t}(x)\right)\otimes \left(x-\mathbb{E}_{\rho_t}(x)\right)\rho\rd x\,,    
\end{equation}
\begin{equation}\label{eqn:cov_rhoG}
\mathrm{Cov}_{\rho_t,\mathcal{G}}=\Cov_{\mathcal{G},\rho_t}^\top=\int_{\mathbb{R}^d}\left(x-\mathbb{E}_{\rho_t}(x)\right)\otimes \left(\mathcal{G}(x)-\mathbb{E}_{\rho_t}(\mathcal{G}(x))\right)\rho(t,x)\rd x\,.
\end{equation}

Similarly, for an ensemble of particles $\{x_n\}^N_{n=1}$, we define the empirical means and covariance matrices
\begin{equation}\label{eqn:def_ensemble_cov}
\begin{aligned}
&\bar{x}=\frac{1}{N}\sum^N_{i=1}x_i,\quad \bar{\mathcal{G}}=\frac{1}{N}\sum^N_{i=1}\mathcal{G}(x_i)\,,\\
&\Cov_{x,x}=\frac{1}{N}\sum^N_{i=1}(x_i-\bar{x})\otimes (x_i-\bar{x})\,,\quad \text{and}\quad \Cov_{x,\mathcal{G}}=\frac{1}{N}\sum^N_{i=1}(x_i-\bar{x})\otimes (\mathcal{G}(x_i-\bar{\mathcal{G}})\,.    
\end{aligned}
\end{equation}

\subsection{Ensemble Kalman inversion}\label{sec:EKI}
This section reviews the EKI method~\cite{Iglesias_2013} and summarizes its results of the mean-field analysis~\cite{ding2019ensemble}.

\subsubsection{Algorithm}

Ensemble Kalman Inversion (EKI) is an algorithm that designs a dynamics aiming at driving the prior distribution to the posterior distribution in a finite time horizon. More specifically, a PDE will be designed so that its solution has the following form.
\begin{equation}\label{actualmu}
\rho(t,x)\propto \rho_{\text{prior}}(x)\exp\left(-t\Phi(x;y)\right)\,.
\end{equation}
Recall Bayes' theorem~\eqref{eqn:bayes}, it is immediate that
\[
\rho(t=0,x)=\rho_{\text{prior}}(x)\,,\quad \rho(t=1,x)=\rhotarg(x)\,.
\]
One proposal to achieve this is to implement the following PDE:
\begin{equation}\label{eqn:EKI_PDE}
\partial_t\rho=\mathcal{L}[\rho]=-\nabla_x\cdot\left(\left(y-\mathcal{G}(x)\right)^\top \Gamma^{-1}\mathrm{Cov}_{\mathcal{G},\rho_t}\rho\right)+\frac{1}{2}\mathrm{Tr}\left(\mathrm{Cov}_{\rho_t,\mathcal{G}}\Gamma^{-1}\mathrm{Cov}_{\mathcal{G},\rho_t}\mathcal{H}_x(\rho)\right)
\end{equation}
where $\mathcal{H}_x$ denotes the hessian in $x$.

Therefore, as a consequence, a Monte Carlo type solver that simulates~\eqref{eqn:EKI_PDE} provides samples drawn from the target distribution $\rhotarg$.

\begin{remark}
It is typically overlooked in the literature that in the most general setting,~\eqref{actualmu} does not solve~\eqref{eqn:EKI_PDE} unless the forward map is linear. Namely, only when $\mathcal{G}(x)=Ax$ and thus
$\Phi\left(x;y\right)=\frac{1}{2}\left(x^\dagger-x\right)^\top A^\top \Gamma^{-1}A(x^\dagger-x)$ where we denote $y=Ax^\dagger$,~\eqref{actualmu} solves~\eqref{eqn:EKI_PDE}, and thus the PDE solver generates samples from $\rhotarg$ at $t=1$. For general nonlinear $\mathcal{G}$, this is only an approximation.

Indeed, as proved in~\cite{ding2020ensemblecorrect}, the correct PDE, besides having the $\mathcal{L}$ operator defined as in~\eqref{eqn:EKI_PDE}, should also make use of a correction term:
\begin{equation*}\label{eqn:muPDE_nonlinear}
\partial_t\rho(t,x)=\mathcal{L}\left[\rho(t,x)\right]+\mathcal{R}(t,x;\rho)\rho(t,x)\,,
\end{equation*}
where $\mathcal{R}(t,x;\rho)$ has a complicated nonlinear form that gets expanded out in~\cite{ding2020ensemblecorrect}. This form strongly suggests that a sampling strategy should have all particles weighted with the weight adjusted according to $\mathcal{R}$.

Depending on the size of $\mathcal{R}$, EKI may produce samples that are $O(1)$ away from the target distribution $\rhotarg$. It is still an open problem to estimate and control this weight term $\mathcal{R}$, but some discussions can be found in~\cite{Ernst,ding2020ensemblecorrect,law_tembine_tempone}.
\end{remark}

According to~\cite{ksendal2003}, the SDE representation of \eqref{eqn:EKI_PDE}, the associated System B is:
\begin{equation}\label{eqn:EKI_ancillary}
\rd \tilde{x}_i(t)=\Cov_{\rho_t,\mathcal{G}}\Gamma^{-1}\left(y-\MCG(\tilde{x}_i(t))\right)\rd t+\Cov_{\rho_t,\mathcal{G}}\Gamma^{-\frac{1}{2}}\rd W^i_t\,.
\end{equation}

As discussed in Section~\ref{sec:unif_analysis}, System B cannot be realized numerically. To turn it into an algorithm, the groundtruth covariance $\Cov_{\rho_t,\mathcal{G}}$ is replaced by the ensemble covariance matrix $\Cov_{x_t,\mathcal{G}_t}$, and we obtain the new system, the associated System C:
\begin{equation}\label{eqn:SDE_general}
\rd x_i(t)=\Cov_{x_t,\mathcal{G}_t}\Gamma^{-1}\left(y-\MCG(x_i(t))\right)\rd t+\Cov_{x_t,\mathcal{G}_t}\Gamma^{-\frac{1}{2}}\rd W^i_t\,,
\end{equation}
where the covariance matrices are self-generated defined in~\eqref{eqn:def_ensemble_cov}.

Both System B and System C set the samples at the initial time $\{x_i(t=0)\}$ i.i.d. drawn from the prior distribution $\rho(t=0,\cdot)=\rho_{\mathrm{prior}}$. 

The algorithm of EKI is to further discretize system C using the Euler-Maruyama method with time step $h>0$, as summarized in Algorithm \ref{alg:EKI}. The algorithm is terminated at $M=1/h$ time step to ensure that the system runs to the pseudo-time $t=1$.

\begin{algorithm}[htb]
\caption{\textbf{Ensemble Kalman Inversion}}\label{alg:EKI}
\begin{algorithmic}
\State \textbf{Preparation:}
\State 1. Input: $N\gg1$; $h$ (Stepsize); $M=1/h$ (stopping index); $\Gamma$; and $y$ (data).
\State 2. Initial: $\{x^0_i\}^N_{i=1}$ sampled from initial distribution $\rho_\ini$.

\State \textbf{Run: } Set time step $m=0$;
\State \textbf{While} $m<M$:

\noindent 1. Define empirical means and covariance:
\begin{equation*}\label{eqn:ensemble_alg}
\begin{aligned}
\quad&\overline{x}^m=\frac{1}{N}\sum^N_{i=1}x^m_i\,, \overline{\MCG}^m=\frac{1}{N}\sum^N_{i=1}\MCG(x^m_i)\ \text{and} \ \Cov_{x^m,\mathcal{G}^m}=\frac{1}{N}\sum^N_{i=1}\left(x^m_i-\overline{x}^m\right)\otimes \left(\MCG(x^m_i)-\overline{\MCG}^m\right)\,.
\end{aligned}
\end{equation*}

\noindent2. Update ensemble particles ($\forall 1\leq i\leq N$)
\begin{equation*}\label{eqn:update_uujn}
\begin{aligned}
&\quad x^{m+1}_{\ast,i}=x^m_i+h\Cov_{x^m,\mathcal{G}^m}\Gamma^{-1}(y-\MCG(x^m_i))\,,\\
&\quad x^{m+1}_i=x^{m+1}_{\ast,i}+\sqrt{h}\Cov_{x^m,\mathcal{G}^m}\Gamma^{-1/2}\xi^{m}_i,\quad \text{with}\quad \xi^{m}_i\sim\mathcal{N}(0,\mathrm{I})\,.
\end{aligned}
\end{equation*}
\noindent 3. Set $m\to m+1$
\State \textbf{end}
\State \textbf{Output:} Ensemble particles $\{x^M_i\}^N_{i=1}$.
\end{algorithmic}
\end{algorithm}

\subsubsection{Mean-field analysis for EKI}
The mean-field derivation of EKI amounts to equating~\eqref{eqn:SDE_general} with its underlying PDE~\eqref{eqn:EKI_PDE}. The following theorem summarizes the result.
\begin{theorem}[Mean-field limit for EKI]\label{thm:mean_fieldEKI}
Suppose that $\mathcal{G}$ satisfies the weak nonlinear assumption\footnote{For technical reasons, we can only prove the mean-field under the weak nonlinearity assumption. In particular, we assume that $\mathcal{G}$ is weak nonlinear, meaning that there is a matrix $A\in\mathbb{R}^{d\times d}$ such that $\mathcal{G}(x)=Ax+\mathrm{b}(x)$ where $\mathrm{b}(x):\mathbb{R}^d\to\mathbb{R}^{d}$ is a smooth bounded function satisfying $\text{Range}(\mathrm{b})\perp_{\Gamma^{-1}}\text{Range}(A),\ \left|\mathrm{b}(x)\right|+\left|\nabla_x\mathrm{b}(x)\right|\leq B$ with some constant $B>0$ in $\mathbb{R}^d$, and $a\perp_{\Gamma^{-1}}b$ means $a^\top \Gamma^{-1}b=0$ and $a^\top$ is to take transpose of $a$.}. Denote $\rho_{N}(t,x)$ the ensemble distribution of $\{x_i(t)\}$ as defined in~\eqref{eqn:approx} and $\rho(t,x)$ the weak solution to \eqref{eqn:EKI_PDE} with $\rho(0,x)=\rho_{\mathrm{prior}}$, then $\rho_N\to\rho$ in the Wasserstein sense and the weak sense for all finite time $t$:
\begin{itemize}
    \item For any $\epsilon>0$, there is a constant $C_\epsilon(t)$ independent of $N$ such that: 
\begin{equation}\label{eqn:decayW2}
\mathbb{E}\left(W_2(\rho_{N}(t,x),\rho(t,x))\right)\leq C_\epsilon(t)
\left\{
\begin{aligned}
&N^{-\frac{1}{4}},\ d\leq 4\\
&N^{-1/d},\ d>4
\end{aligned}\,.
\right.
\end{equation}
\item For any $l$-Lipschitz test function $g$, for any $\epsilon>0$, there is a constant $C_\epsilon$ depending on $\epsilon$, $l$ and $t$ and independent of $N$ such that:
\begin{equation}\label{weakconvergence}
\left(\Eb\left|\int g(x)\left[\rho_N(t,x)\rd x-\rho(x,t)\rd x\right] \right|^2\right)^{\frac{1}{2}}
 \leq C_\epsilon N^{-\frac{1}{2}+\epsilon}\,.
\end{equation}
\end{itemize}
\end{theorem}
We note the convergence is in expectation sense. Different realizations of Brownian motion provides different configuration of $\rho_N$. In expectation, $\rho_N$ convergences to $\rho$, and the expectation sign takes expectation over the randomness from Brownian motions~\eqref{eqn:approx}.

The theorem clearly suggests that the Wasserstein convergence rate sees the curse of dimensionality. For $d\geq 4$, the rate is slower than the classical Monte Carlo rate of $\frac{1}{\sqrt{N}}$. We believe that the rate is optimal within the current machinery, see~\cite{Fournier2015}. On the other hand, the convergence in the weak sense (tested by a Lipschitz function) is always as strong as one would expect from an MC sampling. It should be noted that when $g(x)=x^p$ for any $p\in\mathbb{N}_+$, we expect convergence of the moments\footnote{$g(x)=x^p$ is not global Lipschitz, so a strategy that regularizes the test function through truncation is necessary to deploy~\eqref{weakconvergence}. We leave out details.}.

The main strategy is already discussed in Section \ref{sec:unif_analysis}: the proof boils down to showing both~\eqref{eqn:FG_convergence} and~\eqref{eqn:coupling_error}, as presented below.

Firstly, we note that \eqref{eqn:FG_convergence} is proved by calling~\cite[Theorem 1]{Fournier2015}. In the current context, 
\begin{lemma}\label{lem:vj_FP}
Let $\{\tilde{x}_i(t)\}$ solve \eqref{eqn:EKI_ancillary} with $\rho(t,x)$ being the solution to~\eqref{eqn:EKI_PDE}. Further denote $\tilde{\rho}_N(t,x)$ the ensemble distribution of $\{\tilde{x}_i(t)\}$ as defined in \eqref{eqn:rho_n_tilde}, then, under the condition in Theorem~\ref{thm:mean_fieldEKI}, there is a constant $C(t)$ independent of $N$ such that for all finite $t$:
\begin{equation}\label{eqn:vj_FP2}
\mathbb{E}\left(W^2_2(\widetilde{\rho}_N(t,x),\rho(t,x))\right)\leq C(t)N^{-\min\{1/2,2/d\}}\,.
\end{equation}
\end{lemma}

Now, we prove \eqref{eqn:coupling_error}. It should be noted that the calculation in~\eqref{eqn:velocity_control} does not apply directly: we do not have the Lipschitz continuity against $\rho$. In~\cite{ding2019ensemble}, a bootstrapping argument was deployed. To be more precise, we define the differences
\[
\Delta_i(t)=x_i(t)-\tilde{x}_i(t)\,.
\]
By subtracting the two equations~\eqref{eqn:EKI_ancillary} and~\eqref{eqn:SDE_general}, a conditional convergence can be shown, namely:
\begin{lemma}\label{lem:small}
If there is $\gamma\geq 0$ so that
\[
\mathbb{E}(|\Delta_i(t)|^2)\leq N^{-\gamma}\,,\quad\forall i\,,
\]
then for any $\epsilon>0$, there is a constant $C_\epsilon$ so that
\[
\mathbb{E}(|\Delta_i(t)|^2)\leq C_\epsilon N^{-1/2-\gamma/2+\epsilon}\,,\quad\forall i\,.
\]
\end{lemma}

Note that
\[
\frac{1}{2}+\frac{\gamma}{2}-\epsilon>\gamma\,,\quad\text{for}\quad \gamma<\frac{1}{2}-\epsilon\,,
\]
so conditioned on a lower convergence rate, the lemma can be deployed to push the convergence to a faster rate. To kick off this bootstrapping, we note that by setting $\gamma=0$, one needs the boundedness of the second moment, which can be proved. Then Lemma~\ref{lem:small} gets repeatedly deployed to finally push the results to:
\begin{equation}\label{eqn:bootstrap}
\mathbb{E}(|\Delta_i(t)|^2)\leq C_\epsilon N^{-1+\epsilon}\quad\Rightarrow\quad \mathbb{E}(W_2(\tilde{\rho}_N\,,\rho_N))\leq C_\epsilon N^{-1/2+\epsilon/2}\,,
\end{equation}
for any $\epsilon>0$.

Combining it with Lemma~\ref{lem:vj_FP} and choosing $\epsilon=1/2$ in \eqref{eqn:bootstrap}, we obtain:
\[
\begin{aligned}
&\mathbb{E}\left(W_2(\rho_N,\rho)\right)\leq \mathbb{E}\left(W_2(\rho_N,\tilde{\rho}_N)\right)+\mathbb{E}\left(W_2(\tilde{\rho}_N,\rho)\right)\\
\leq & \mathbb{E}\left(W_2(\rho_N,\tilde{\rho}_N)\right)+\left(\mathbb{E}\left(W^2_2(\tilde{\rho}_N,\rho)\right)\right)^{1/2}\leq C
\left\{
\begin{aligned}
&N^{-\frac{1}{4}},\quad d\leq4\\
&N^{-1/d},\quad d>4\\
\end{aligned}
\right.\,,
\end{aligned}
\]
concluding \eqref{eqn:decayW2}.

The weak convergence in~\eqref{weakconvergence} is a classical practice. First, noting the triangle inequality, we have
\begin{equation}\label{uf1}
\begin{aligned}
&\left(\Eb\left|\int g(x)\left[\rho_N\rd x-\rho\rd x\right] \right|^2\right)^{\frac{1}{2}}\\
\leq &\left(\Eb\left|\int g(x)\left[\rho_N\rd x-\tilde{\rho}_N\rd x\right] \right|^2\right)^{\frac{1}{2}}+\left(\Eb\left|\int g(x)\left[\tilde{\rho}_N\rd x-\rho\rd x\right] \right|^2\right)^{\frac{1}{2}}\,.
\end{aligned}
\end{equation}
These two terms can be controlled separately. To control the first term, we note
\begin{equation}\label{uf2}
\begin{aligned}
&\Eb\left|\int g(x)\left[\rho_N\rd x-\tilde{\rho}_N\rd x\right] \right|^2=\EE\left|\frac{1}{N}\sum^N_{i=1}g(x_i(t))-g(\tilde{x}_i(t))\right|^2\\
\leq&\frac{l^2}{N^2}\EE\left(\sum^N_{i=1}|x_i(t)-\tilde{x}_i(t)|^2\right)\leq C_\epsilon l^2N^{-1+\epsilon}\,,
\end{aligned}
\end{equation}
where the $l$-Lipschitz and H\"older's inequality are used in the first inequality, and~\eqref{eqn:bootstrap} is used in the second inequality. The control of the second term is essentially the law of large numbers:
\[
\EE\left|\int g(x)\left[\tilde{\rho}_N\rd x-\rho(t,x)\rd x\right]\right|^2
=\EE\left|\frac{1}{N}\sum^N_{i=1}g(\tilde{x}_i(t))-\EE_{\rho_t}(g)\right|^2=\frac{1}{N}\text{Var}_\rho[g]
\]
where $\mathrm{Var}_{\rho_t}[g]=\EE_{\rho}[g^2]-\left(\EE_{\rho}[g]\right)^2$. The term is bounded due to the $l$-Lipschitz property of $g$.
\begin{equation}\label{vf1}
\EE\left|\int g(x)\left[\widetilde{\rho}_N\rd x-\rho(t,x)\rd x\right]\right|^2\leq C\left(l\right)N^{-1}\,.
\end{equation}

Combining the two terms into \eqref{uf1}, \eqref{weakconvergence} is proved.

\subsection{Ensemble Kalman sampler}\label{sec:EKS}

This section is dedicated to the mean-field analysis of the algorithm Ensemble Kalman sampler (EKS). As for EKI, we first present the algorithm and then show the application of the procedure as demonstrated in subsection~\ref{sec:unif_analysis}.

\subsubsection{Algorithm}
EKS was proposed in~\cite{EKS}, and to some extent can be viewed as a modification to the overdamped Langevin dynamics. Unlike Langevin, which induces a single-particle MCMC type algorithm, EKS introduces the communications between particles and induces an ensemble method.

There are many approaches to justify the success of the overdamped Langevin. One approach is to view it as a particle method for the Fokker-Planck equation, and the beautiful discovery~\cite{JKO_1998} reformulates it as the gradient flow of relative entropy. More specifically, the Fokker-Planck PDE writes as:
\begin{equation}\label{eqn:FP}
\frac{\rd \rho}{\rd t}=\nabla_x\cdot\left(\nabla_x f(x;y)\rho+\nabla_x\rho\right) = -\nabla_{W_2}\KL(\rho|\rhotarg)\,.
\end{equation}
Here $\KL(\mu|\nu) = \int\log(\rd\mu/\rd\nu)\rd\mu$ is termed relative entropy, or sometimes also called the Kullback-Leibler divergence. In the equation, $\nabla_{W_2}\KL$ stands for the gradient of $\rho$ of the functional $\KL$ against the Wasserstein metric. Since it is a gradient flow type PDE, the $\KL$ monotonically decreases in time. Moreover, under proper assumptions on the function $f(x;y)$ (log-concave for example), when $t\rightarrow\infty$, $\rho(t,x)$ converges to $\rhotarg\propto\exp(-f(x;y))$~\cite{Markowich99onthe}. 

On the discrete level, to simulate the equation, one can deploy the particle method. It is then run according to the following Langevin dynamics:
\begin{equation}\label{eqn:ol}
    \rd x(t)=-\nabla_x f(x;y)\rd t+\sqrt{2}\rd W_t\,.
\end{equation}
As a natural consequence, with long enough time, the particle produced by~\eqref{eqn:ol} can be seen as a sample drawn from the target distribution $\rhotarg$. This classical sampling algorithm is called overdamped Langevin Monte Carlo~\cite{doi:10.1063/1.436415,PARISI1981378,roberts1996}.

The Ensemble Kalman Sampler (EKS) can be seen as an ensemble modification of \eqref{eqn:ol}. Indeed, it rewrites~\eqref{eqn:FP} into the following form:
\begin{equation}\label{FKPK}
\partial_t\rho=\nabla\cdot(\rho\Cov_{\rho_t}\nabla_xf(x;y))+\Tr\left(\Cov_{\rho_t}\mathcal{H}_x\rho\right)\,,
\end{equation}
where $\Cov_{\rho_t}$ is defined in \eqref{eqn:cov_rho}.

One can further formulate it as a gradient flow over a specially designed metric~\cite[Section 3.3]{EKS} to analytically quickly obtain convergence in time. We omit the details here.

The stark difference between this modified Fokker-Planck equation~\eqref{FKPK} and the classical Fokker-Planck equation~\eqref{eqn:FP} is that this new equation is nonlinear. When discretized using particle methods, one needs to differentiate System B and System C, as discussed in Section~\ref{sec:unif_analysis}. System B is a collection of particles that are passively pushed by the underlying flow, and in this context becomes:
\[
    \rd \tilde{x}_i=-\Cov_{\rho_t}\nabla_x f(\tilde{x}_i;y)\rd t+\sqrt{2\Cov_{\rho_t}}\rd W^i_t\,.
\]
In the Bayesian setting, $f(x;y)$ adopts the form in~\eqref{eqn:cost_bayes} and the formulation rewrites as:
\begin{equation}\label{eqn:ensemble_FKPK}
    \rd \tilde{x}_i=-\Cov_{\rho_t}\nabla_x \MCG(x_i)\Gamma^{-1}\left(\MCG(\tilde{x}_i(t))-y\right)\rd t-\Cov_{\rho_t}\Gamma^{-1}_0\left(\tilde{x}_i(t)-x_0\right)\rd t+\sqrt{2\Cov_{\rho_t}}\rd W^i_t\,.
\end{equation}

It is clear that the covariance in the above formulation is the ground truth covariance computed from the underlying flow $\rho$ and is not available, so in real computation, the ensemble covariance is deployed as an approximation. This leads to System C:
\begin{equation}\label{eqn:particle_FKPK}
        \rd x_i=-\mathrm{Cov}_{x_t,x_t}\nabla_x \MCG(x_i)\Gamma^{-1}\left(\MCG(x_i(t))-y\right)\rd t-\Cov_{x_t,x_t}\Gamma^{-1}_0\left(x_i(t)-x_0\right)\rd t+\sqrt{2\mathrm{Cov}_{x_t,x_t}}\rd W^i_t\,.
\end{equation}
where $\Cov_{x_t,x_t}$ is defined in~\eqref{eqn:def_ensemble_cov}.

There are two nice features of this modified Fokker-Planck equation and the ensemble discretization pair, in comparison to the original Fokker-Planck equation and the Langevin dynamics pair. As discovered in~\cite{Carrillo_2021}, the modified Fokker-Planck~\eqref{FKPK} is more stable against the conditioning of the landscape $f$. More specifically, when $\rhotarg$ is a Gaussian distribution, the convergence of $\rho(t)$ to $\rhotarg$ in the Wasserstein-2 distance is exponential, and the rate of convergence is independent of the conditioning of the target Gaussian.

The second advantage of deploying this new equation is that, when the forward map $\mathcal{G}$ is almost linear, we are in a very good position to transition the problem into a gradient-free formulation. More specifically, by formulating
\[
\Cov_{\rho_t}\nabla_x \MCG(x(t))\quad\to\quad\mathrm{Cov}_{\rho_t,\mathcal{G}_t}\,,
\]
then System B as in written in~\eqref{eqn:ensemble_FKPK} takes a gradient-free form:
\begin{equation}\label{eqn:EKS_SDE}
    \rd \tilde{x}_i=-\mathrm{Cov}_{\rho_t,\mathcal{G}_t}\Gamma^{-1}\left(\MCG(\tilde{x}_i)-y\right)-\Cov_{\rho_t}\Gamma^{-1}_0\left(\tilde{x}_i-x_0\right)\rd t+\sqrt{2\Cov_{\rho_t}}\rd W^i_t \,,
\end{equation}
and System C, as in~\eqref{eqn:particle_FKPK}, becomes:
\begin{equation}\label{Ecov_dis_revise}
    \rd x_i=-\mathrm{Cov}_{x_t,\mathcal{G}_t}\Gamma^{-1}(\mathcal{G}(x_i)-y)\rd t-\mathrm{Cov}_{x_t,x_t}\Gamma^{-1}_0(x_i(t)-x_0)\rd t+\sqrt{2\mathrm{Cov}_{x_t,x_t}}\rd W^{i}_t\,,
\end{equation}
where $\Cov_{x_t,\mathcal{G}_t}$ is the ensemble covariance between $\{x_i(t)\}$ and $\{\mathcal{G}(x_i(t))\}$ defined in~\eqref{eqn:def_ensemble_cov}. Note that the computation does not require the access of $\nabla\mathcal{G}$, a desirable feature in many scientific domains~\cite{DuIaXi:2019turbulence,GuHaScHuDuLoWu:2023posteriori}. We summarize in Algorithm \ref{ALG1} the numerical implementation of equation~\eqref{Ecov_dis_revise}.

For mathematical justification of the algorithm, we provide the mean-field analysis in the theorem below.

\begin{algorithm}[h]
\caption{\textbf{Ensemble Kalman sampling}}\label{ALG1}
\begin{algorithmic}
\State \textbf{Preparation:}

\State 1. Input: $N\gg1$; $h$ (stepsize); $M$ (stopping index); $\Gamma$; $\Gamma_0$; and $y$ (data).
\State 2. Initial: $\{x^0_i\}^N_{i=1}$ sampled from an initial distribution $\rho_\ini$.

\State \textbf{Run: } Set time step $m=0$;
\State \textbf{While} $m<M$:

1. Define empirical means and covariance:
\[
\begin{aligned}
\quad&\overline{x}^m=\frac{1}{N}\sum^N_{i=1}x^m_i\,,\ \text{and}\ \overline{\MCG}^m=\frac{1}{N}\sum^N_{i=1}\MCG(x^m_i)\,,\\
\quad&\Cov_{x^m,x^m}=\frac{1}{N}\sum^  N_{i=1}\left(x^m_i-\overline{x}^m\right)\otimes \left(x^m_i-\overline{x}^m\right) \ \text{and} \ \Cov_{x^m,\mathcal{G}^m}=\frac{1}{N}\sum^N_{i=1}\left(x^m_i-\overline{x}^m\right)\otimes \left(\MCG(x^m_i)-\overline{\MCG}^m\right)\,.
\end{aligned}
\]

2. Update ensemble particles ($\forall 1\leq i\leq N$)
\begin{equation*}\label{eqn:update_ujn}
\begin{aligned}
&x^{m+1}_{*,i}=x^m_i-h\Cov_{x^m,\mathcal{G}^m}\Gamma^{-1}\left(\MCG(x^m_i)-y\right)-h\Cov_{x^m,x^m}\Gamma^{-1}_0\left(x^{m+1}_{*,i}-x_0\right)\,,\\
&x_n^{m+1}=x^{m+1}_{*,i}+\sqrt{2h\Cov_{x^m,x^m}}\xi^m_i\,,\quad\text{with}\quad\xi^m_{i}\sim \mathcal{N}(0,\mathrm{I})\,.
\end{aligned}
\end{equation*}

3. Set $m\to m+1$.
\State \textbf{end}
    \State \textbf{Output:} Ensemble particles $\{x^M_i\}^N_{i=1}$.
\end{algorithmic}
\end{algorithm}

\begin{theorem}[Mean-field limit of EKS]\label{thm:main_EKS}
Suppose $\mathcal{G}(x)=Ax$ with $A\in\mathbb{R}^{d\times d}$. Let $\rho(t,x)$ be the weak solution to the Fokker-Planck equation and $\{x_i(t)\}$ be a solution to~\eqref{Ecov_dis_revise} with $x_i(0)$ i.i.d. drawn from the distribution $\rho_\pri$. Assume:
\begin{equation}\label{thmcondition}
\lambda_{\min}\left(\Cov_{\rho_t}\right)> \lambda_{\max}\left(\mathrm{Cov}_{\rhotarg}\right),\quad \forall 0\leq t\leq T\,,
\end{equation}
then for any $0<\epsilon<1/2$, there exits $C$, depending on $T$ and $\epsilon$, but independent of $N$ such that
\[
\begin{aligned}
\mathbb{E}\left(W_2(\rho_N(t,x),\rho(t,x))\right)\leq C
\left\{
\begin{aligned}
&N^{-\frac{1}{4}},\ d\leq 4\\
&N^{-1/d},\ d>4
\end{aligned}
\right.\,.
\end{aligned}
\]
In addition, given any $l$-Lipschitz function $g$, for any $\epsilon>0$, there is a constant $C_\epsilon$ such that for any $t<\infty$
\[
\begin{aligned}
\left(\Eb\left|\int g(x)\left[\rho_N(t,x)\rd x-\rho(t,x)\rd x\right] \right|^2\right)^{\frac{1}{2}}\leq C_\epsilon N^{-\frac{1}{2}+\epsilon}
\end{aligned}\,.
\]
where the constant $C_\epsilon$ depends on $l$, $\epsilon$ and time $t$ and is independent of $N$ 
\end{theorem}
The proof of the above theorem is similar to that of Theorems \ref{thm:mean_fieldEKI}, so we will not include it here. We refer interested readers to~\cite{Zhi_Qin_2021}.

We note that the assumption~\eqref{thmcondition} is for a technical purpose that we hope to be able to remove. In the proof we deploy the Ando-Hemmen inequality that studies the differences between two matrices when square roots are taken. As long as some other reasonable estimates, which avoid direct use of Ando-Hemmen, can be found to control the sensitivity of taking square roots of matrices, we see the possibility of removing this assumption. In addition, when $\mathcal{G}$ is weakly nonlinear, we anticipate that convergence to the mean-field limit is still valid.
\section{Boltzmann simulator}\label{sec:Boltzmann}
We switch gears to a new set of sampling strategy called the Boltzmann simulator. In contrast to global interactions, the Boltzmann simulator deploys local interactions, drawing inspiration from the Vlasov-Boltzmann equation. 

To facilitate our discussion, we begin by providing an overview of the Vlasov-Boltzmann equation in Section~\ref{sec:boltzmann_idea}. This gives us the motivation to design sampling algorithms. Following this, in Section~\ref{sec:boltzmann_particle}, we propose two sampling methods derived from the Vlasov-Boltzmann equation, namely the Nanbu simulator and the Bird simulator. The mean-field justification of these two methods will then be discussed in Section~\ref{sec:boltzmann_mean_field}. Numerical results are presented in Section~\ref{sec:boltzmann_numerics}.

\subsection{Vlasov-Boltzmann and its properties}\label{sec:boltzmann_idea}

The Boltzmann simulator should be viewed as a particle method for the Vlasov-Boltzmann equation. In the domain $(x,v)\in\Rb^{2d}$:
\begin{equation}\label{eqn:boltzmann}
    \partial_t \mu_t(x,v) + \Lc \mu_t(x,v)
    = \Qc \mu_t(x,v)\,,
\end{equation}
where the transport operator $\Lc$ is defined by
\begin{equation}
\Lc \mu_t(x,v) = v\cdot\nabla_x \mu_t(x,v) - \nabla_x f(x) \cdot \nabla_v \mu_t(x,v) \,,
\end{equation}
and the collision operator $\Qc$ is defined by
\begin{equation}
\Qc \mu_t (x,v) = \iint_{\Sb^{d-1}\times\Rb^{d}} q(v,w,n) (\mu_t(x,v^\ast)\mu_t(x,w^\ast) - \mu_t(x,v)\mu_t(x,w)) \, \rmd n \rmd w \,.
\end{equation}
The transport operator $\Lc$ is the Vlasov component of the equation driven by the Hamiltonian:
\begin{equation}\label{eqn:hamiltonian}
\dot{x} = v\,,\quad\dot{v} = -\nabla_xf
\end{equation}
where $f(x)$ is a given potential function. 
The Boltzmann collision operator $\mathcal{Q}$ describes the interaction between two particles located at the same position $x$ and having pre-collisional velocities $v$ and $w$. During this interaction, they exchange momentum and energy, resulting in new velocities $v^\ast$ and $w^\ast$ after the collision.
The vector $n$ is a unit vector that describes in which direction the particles are scattered off to:
\begin{equation}\label{eqn:new_velocities}
    v^\ast = v + ((w-v)\cdot n) n\,, \quad w^\ast = w + ((v-w)\cdot n) n \,.
\end{equation}
It can be verified that the relative velocity before and after collision is unchanged:
\begin{equation}\label{eqn:velocity_diff}
|v-w| = |v^\ast-w^\ast| \,.
\end{equation}
The term $q(v,w,n)$ is called the cross-section. 
It quantifies the likelihood of such interactions. We have assumed that the medium is homogeneous, meaning $q$ is independent of $x$, and only depends on $|v-w|$ and $(v-w)\cdot n$. 
This assumption is useful when we derive the weak form PDE in Section~\ref{sec:particle_weak_form}. We usually say that the Boltzmann operator is local because all particle interactions take place when their spatial coordinates agree -- thus local in $x$.

This classical local Boltzmann collision operator in equation~\eqref{eqn:boltzmann} can sometimes be relaxed to allow particles at different spatial locations to interact as well. We consider the mollified Vlasov-Boltzmann equation~\cite{GrMe:1997stochastic,Me:1996asymptotic}
\begin{equation}\label{eqn:mollified_boltzmann}
    \partial_t \mu_t(x,v) + 
    \Lc \mu_t (x,v)
    = \widetilde{\Qc} \mu_t (x,v)\,,
\end{equation}
where the collision operator
\begin{equation}
\widetilde{\Qc} \mu_t (x,v) = \iint_{\Sb^{d-1}\times\Rb^{2d}} ~\widetilde{q}(x,v,y,w,n) (\mu_t(x,v^\ast)\mu_t(y,w^\ast) - \mu_t(x,v)\mu_t(y,w)) \, \rmd n \rmd w \rmd y \,.
\end{equation}
The new cross-section $\widetilde{q}(x,v,y,w,n) = I(|x-y|)q(v,w,n)$ is delocalized, with $I(|x-y|)$ serving as a mollifier. It is nonnegative and normalized, i.e. $\int_{\Rb^d} I(|x|) \rmd x = 1$. This equation allows a particle at location $x$ with velocity $v$ to interact with a particle at location $y$ with velocity $w$ to obtain new velocities $v^\ast$ and $w^\ast$. The intensity of such interaction is upper bounded by $I$. Note that the classical Boltzmann~\eqref{eqn:boltzmann} can be recovered if $I=\delta$.

The equation can also be interpreted with a probabilistic viewpoint, dating back to Tanaka~\cite{Ta:1978probabilistic,Ta:2002stochastic}, where he connected the solution of the Boltzmann equation to a Piecewise Deterministic Markov Process (PDMP). 
In this context, we view the solution of the Boltzmann equation as the distribution of particles that evolve according to the force field $-\nabla f(x)$. Each particle is equipped with a Poisson clock, and when the clock rings, the particle, with coordinate $(x,v)$, chooses another particle $(y,w)$ to interact with. The particle being chosen is selected at random from an ocean of infinitely many independent particles distributed according to the same distribution. 
The particle $(x,v)$ collides with the particle $(y,w)$ by randomly sampling a normal direction $n\sim\widetilde{q}(x,v,y,w,n)\rmd n$. 
After the collision, the particle $(x,v)$ adopts a new coordinate $(x,v^\ast)$ and forgets all knowledge about the particle $(y,w)$. It then continues to evolve under the force field. For mathematical convenience, it is further assumed the cross-section is bounded:
\begin{assumption}[Boundedness of Cross-Section]\label{assumption:kernel_bound}
There exists $\Lambda>0$ such that
\begin{equation}
\sup_{x,y,v,w\in\Rb^d} \int_{\Sb^{d-1}} \widetilde{q}(x,v,y,w,n)\rmd n \leq \Lambda <\infty \,.
\end{equation}
\end{assumption}

According to the PDE analysis perspective, there are two nice features of this modified Boltzmann equation.
\begin{itemize}
    \item[--] Equilibrium state is the target distribution. As will be shown in Theorem~\ref{thm:equilibrium}, the equilibrium of the equation is the following:
    \begin{equation}\label{eqn:equilibrium}
\mu^\ast \propto \exp\left(-f(x) - \frac{1}{2}|v|^2\right) \,,
\end{equation}
meaning its $x$-marginal distribution truly recovers the target distribution~\eqref{def:target}.
    \item[--] Long time limit convergences to the equilibrium. With proper conditions, as discussed in Theorem~\ref{thm:boltzmann_equilibrium}, one can also show that, in a weighted $L_\infty$-norm, the PDE drives the initial distribution to the equilibrium state~\eqref{eqn:equilibrium}.
\end{itemize}
The goal of sampling is to design a dynamics whose invariant measure is the target distribution, so that samples following this dynamics, in long time, can be viewed as a drawing from the target distribution. The two features above suggest that the dynamics provided by the Vlasov-Boltzmann system~\eqref{eqn:mollified_boltzmann} truly drives some arbitrary initial data to the target distribution and thus is a feasible choice. If one can design a particle method to simulate~\eqref{eqn:mollified_boltzmann}, the particles, in the long time limit, can be viewed as samples drawn from the target distribution.

We should also note that under the boundedness assumption of the cross-section, the well-posedness of the mollified Vlasov-Boltzmann equation is known~\cite{Gr:1992nonlinear}, so the existence and uniqueness of the solution is not a concern.

We now present the mathematically rigorous description of the two features discussed above. Firstly, the target distribution is an invariant measure of the Vlasov-Boltzmann equation.
\begin{theorem}\label{thm:equilibrium}
The mollified Vlasov-Boltzmann equation~\eqref{eqn:mollified_boltzmann} admits an equilibrium state $\mu^\ast$ defined in~\eqref{eqn:equilibrium}.
\end{theorem}
\begin{proof}
To show the theorem, we are to plug in $\mu^\ast$ into the equation and demonstrate that both the left-hand side and the right-hand side vanishes.

The left hand side transport term vanishes can be shown by direct computation:
\begin{equation}\label{eqn:eq_transport}
v\cdot\nabla_x \mu^\ast(x,v) - \nabla_x f(x) \cdot \nabla_v \mu^\ast(x,v)
= (-v\cdot\nabla_x f(x)) \mu^\ast(x,v) + (\nabla_x f(x) \cdot v)\mu^\ast(x,v) = 0\,.
\end{equation}

The right-hand side collision term also vanishes. From the conservation of energy upon collision:
\begin{equation}\label{eqn:energy_conservation}
|v^\ast|^2 + |w^\ast|^2 = |v|^2 + |w|^2 \,,
\end{equation}
then
\begin{equation}
\begin{aligned}
&\iint_{\Sb^{d-1}\times\Rb^{2d}} \widetilde{q}(x,v,y,w,n) (\mu^\ast(x,v^\ast)\mu^\ast(y,w^\ast) - \mu^\ast(x,v)\mu^\ast(y,w)) \, \rmd n \rmd w \rmd y\\
&\propto \iint_{\Sb^{d-1}\times\Rb^{2d}} \widetilde{q}(x,v,y,w,n) \exp(-f(x)-f(y)) \\
& \qquad \qquad \qquad \qquad \qquad \qquad \left[ \exp(-|v^\ast|^2 - |w^\ast|^2) - \exp(-|v|^2 - |w|^2) \right] \, \rmd n \rmd w \rmd y = 0\,.
\end{aligned}
\end{equation}
\end{proof}

The stability of the equilibrium state $\mu^\ast$ to the Vlasov-Boltzmann equation~\eqref{eqn:boltzmann} has also been studied, see~\cite{Ki:2014boltzmann,Li:2008trend}. In particular, it is proved in~\cite{Ki:2014boltzmann} that for an potential function $f$ with degenerate coordinate dependence, an initial state converges to $\mu^\ast$ exponentially in a weighted-$L_\infty$ norm:
\begin{theorem}[Theorem 1 in~\cite{Ki:2014boltzmann}]\label{thm:boltzmann_equilibrium}
Assume the potential function $f(x)$ has degenerate coordinate dependence, that is, there exists an $0<n<d$, so that $f(x) = f(x_{n+1},\dots,x_d)$ for any $x=(x_1,\dots,x_d)$. Assume the initial condition $\mu(0,x,v)=\mu_0(x,v)$ satisfies
\begin{equation}\label{eqn:boltzmann_convergence_conditions}
\begin{aligned}
&\int_{\Rb^{2d}} \mu_0(x,v)- \mu^\ast(x,v) \rmd x \rmd v = 0 \,, \quad \text{(Mass condition)} \\
&\int_{\Rb^{2d}} \left( \frac{1}{2}|v|^2 + f(x) \right) (\mu_0(x,v)- \mu^\ast(x,v)) \rmd x \rmd v = 0 \,, \quad \text{(Energy condition)} \\
&\int_{\Rb^{2d}} v_i (\mu_0(x,v) - \mu^\ast(x,v)) \rmd x \rmd v = 0 \,, \quad i = 1,\dots,n,\quad\text{(Momentum condition)} \,,
\end{aligned}
\end{equation}
and is close enough to the equilibrium state $\mu^\ast(x,v)$. 
Then under mild regularity conditions and periodic boundary conditions, the Vlasov-Boltzmann equation has a unique global solution $\mu_t(x,v)$. 
Furthermore, there exists a constant $\lambda = \lambda(q,\|f\|_\infty)>0$ such that
\begin{equation}
  \|\Phi \cdot(\mu_t-\mu^\ast)\|_\infty \leq \Oc(e^{-\lambda t}) \,, \quad \forall t \geq 0\,,
\end{equation}
with the weight function $\Phi(x,v) = \left(\frac{1}{2}|v|^2 + f(x,v)\right)^{\beta/2}$ for $\beta>d/2$.
\end{theorem}

\begin{remark}
Before we dive in to use it as a sampling strategy, a couple of comments are in place.
\begin{itemize}
\item The convergence theorem has no requirement on the convexity of the potential, and hence it holds promise for applications in non-logconcave sampling problems. Indeed, in our numerical section, we will demonstrate one example where we have a double-well potential.
\item There is still one step before turning it directly into a sampling method. To be more precise, Theorem~\ref{thm:boltzmann_equilibrium} states the exponential convergence to the classical Vlasov-Boltzmann~\eqref{eqn:boltzmann}. However, this equation cannot be simulated using the particle method. The particle methods can only simulate the modified Boltzmann equation~\eqref{eqn:mollified_boltzmann}. We do not yet have a convergence-in-time result that shows the solution converging to the equilibrium state for the modified Boltzmann equation.
\item The mathematical reason underpinning the Boltzmann simulator is the fact that the Bhatnagar-Gross-Krook (BGK) operator is typically deployed as an approximation to the collision operator. The linearized Vlasov-BGK equation writes:
\begin{equation*}
    \partial_t \mu_t(x,v) + v\cdot\nabla_x \mu_t(x,v) - \nabla_x f(x) \cdot \nabla_v \mu_t(x,v) = \rho_t M - \mu_t
\end{equation*}
where $M\propto e^{-|v|^2/2}$ is the normalized Maxwellian, and $\rho_t(x) = \int\mu_t(x,v)\rd v$ is the density. The particle simulation of this equation is to have the particle running according to the Hamiltonian flow~\eqref{eqn:hamiltonian}, and scatters off with a new velocity drawn according to Gaussian $M$ upon a Poison clock. This formulation corresponds exactly to the randomized Hamiltonian Monte Carlo algorithm~\cite{BoSa:2017randomized,LuWa:2022explicit}. 
    \end{itemize}
\end{remark}
These nice features suggest that the particle simulator of the Vlasov-Boltzmann equation should give a good sampling method.

\subsection{Nanbu and Bird algorithm}\label{sec:boltzmann_particle}

Both the Nanbu sampler and the Bird sampler are termed Direct Simulation Monte Carlo (DSMC) methods to simulate the mollified Vlasov-Boltzmann equation~\eqref{eqn:mollified_boltzmann}, see~\cite{PaRu:2001introduction}.

According to the probabilistic interpretation mentioned in Section~\ref{sec:boltzmann_idea}, the Poisson clock determines the time for the particle at hand to ``collide'' with peers. However, this interpretation does not say what happens to the collision partner. Depending on what happens to this collision partner, different algorithms have been developed by Nanbu~\cite{Na:1980direct} and Bird~\cite{Bi:1970direct}.

We first introduce the Bird Algorithm. To start, we generated several samples $\{(x_i(0),v_i(0))\}_{i=1}^N$ at the initial time. For each ordered pair $(i,j)$, we attach an independent Poisson process with rate $\Lambda/N$. 
Between two jump times, the particles evolve independently according to the free transport~\eqref{eqn:hamiltonian}.
When the clock of the ordered pair $(i,j)$ rings at time $T_{ij}$, the state of the pair $(i,j)$ is re-evaluated. There are two possibilities:
\begin{itemize}
\item[--] With probability $\Delta_{ij} = \Lambda^{-1}\int_{\Sb^{d-1}}\widetilde{q}(x_i(T_{ij}),v_i(T_{ij}),x_j(T_{ij}),v_j(T_{ij}),n)\rmd n$, the pair is updated. They adopt new velocities according to~\eqref{eqn:new_velocities} with $n$ drawn from:
\[
n\sim \Lambda^{-1}\widetilde{q}(x_i(T_{ij}),v_i(T_{ij}),x_j(T_{ij}),v_j(T_{ij}),n)\rmd n\,.
\]
\item[--] With probability $1-\Delta_{ij}$, the state of the $(i,j)$ pair is unchanged.
\end{itemize}

We summarize the Bird method in Algorithm~\ref{alg:bird}.

\begin{algorithm}[h]
\caption{\textbf{Bird sampling method}}\label{alg:bird}
\begin{algorithmic}
\State \textbf{Preparation:}

\State 1. Input: $N\gg1$; $T$ (terminal time);$\Lambda$ (upper bound for total cross-section); $f(x)$.
\State 2. Initial: $\{(x_i(0),v_i(0))\}_{i=1}^N$ sampled from an initial distribution $\mu_\ini$.

\State \textbf{Run: } Set system time $t=0$; Sample initial collision time $t_{ij}\sim \mathrm{Exp}\left(\Lambda/N\right)$, $i,j=1,\dots,N$.
\State \textbf{While} $t<T$:

1. Pick $\tau_0 = \underset{i,j=1,\dots,N}{\min} t_{ij}$ and $(i_\rmc,j_\rmc) = \underset{i,j=1,\dots,N}{\argmin} \, t_{ij}$; set $\tau_1 = \tau_0\wedge T$.

2. Evolve over $[t, \tau_1]$ Hamilton's equations 
\[
\dot{x}_i = v_i \,, \quad \dot{v}_i = -\nabla_x f(x_i) \,, \quad i=1,\dots,N
\]
with initial condition $(x_i(t),v_i(t))$.

3. If $\tau_0\leq T$, implement the collision stage; else jump to Step 4. 
Sample a number $a$ uniformly over $[0,1]$. If
\[
a \leq \frac{1}{\Lambda} \int_{\Sb^{d-1}} \widetilde{q}(x_i(\tau_1),v_i(\tau_1),x_j(\tau_1),v_j(\tau_1),n) \rmd n \,,
\]
then perform collision: sample the normal direction $n \in \Sb^{d-1}$ by
\[
n \sim \widetilde{q}(x_{i_\rmc}(\tau_1),v_{i_\rmc}(\tau_1),x_j(\tau_1),v_j(\tau_1),n) \rmd n \,,
\]
then update the velocity of particle $i_\rmc$ and $j_\rmc$ according to the rule
\begin{align*}
&v_{i_\rmc} \leftarrow v_{i_\rmc} + ((v_{j_\rmc} - v_{i_\rmc})\cdot n)n \\
&v_{j_\rmc} \leftarrow v_{j_\rmc} + ((v_{i_\rmc} - v_{j_\rmc})\cdot n)n \,;
\end{align*}
else jump to Step 4.

4. Update system time $t\leftarrow\tau_1$ and update the next collision time of the pair $(i_\rmc,j_\rmc)$ by $t_{i_\rmc j_\rmc}\leftarrow t_{i_\rmc j_\rmc} + \delta t$ with $\delta t \sim \mathrm{Exp}\left(\Lambda/N\right)$.

\State \textbf{end}
\State \textbf{Output:} Ensemble particles $\{(x_i(T),v_i(T))\}_{i=1}^N$.
\end{algorithmic}
\end{algorithm}

The distinctive feature of the Nanbu method, compared to the Bird method, is that its Poisson clock is attached to a single particle, and it more views each particle to interact with a ``mean-field'' ocean. Each particle has an independent Poisson clock and evolves according to the Hamiltonian dynamics~\eqref{eqn:hamiltonian} until the clock rings. When one clock rings, the associated particle finds, uniformly among its peers, another particle with which to collide with according to the rate described in $\widetilde{q}$. 
However, only the specific particle whose clock rings undergoes an update.

We summarize the Nanbu method in Algorithm~\ref{alg:nanbu}.

\begin{algorithm}[h]
\caption{\textbf{Nanbu sampling method}}\label{alg:nanbu}
\begin{algorithmic}
\State \textbf{Preparation:}

\State 1. Input: $N\gg1$; $T$ (terminal time);$\Lambda$ (upper bound for total cross-section); $f(x)$.
\State 2. Initial: $\{(x_i(0),v_i(0))\}_{i=1}^N$ sampled from an initial distribution $\mu_\ini$.

\State \textbf{Run: } Set system time $t=0$; Sample initial collision time $t_{i}\sim \mathrm{Exp}(\Lambda)$, $i=1,\dots,N$.
\State \textbf{While} $t<T$:

1. Pick $\tau_0 = \underset{i=1,\dots,N}{\min} t_{i}$ and $i_\rmc = \underset{i=1,\dots,N}{\argmin}\, t_{i}$; set $\tau_1 = \tau_0\wedge T$.

2. Evolve over $[t, \tau_1]$ Hamilton's equations 
\begin{equation*}\label{eqn:boltzmann_ode}
\dot{x}_i = v_i \,, \quad \dot{v}_i = -\nabla_x f(x_i) \,, \quad i=1,\dots,N
\end{equation*}
with initial condition $(x_i(t),v_i(t))$.

3. If $\tau_0\leq T$, implement the collision stage; else jump to Step 4. 
Sample a number $a$ uniformly over $[0,1]$. If
\[
a \leq \frac{1}{N\Lambda} \sum_{j=1}^N \int_{\Sb^{d-1}} \widetilde{q}(x_i(\tau_1),v_i(\tau_1),x_j(\tau_1),v_j(\tau_1),n) \rmd n \,,
\]
then perform collision: uniformly sample an index $j\in\{1,\dots,N\}$, and sample $n\in\Sb^{d-1}$ according to the density $\widetilde{q}(x_{i_\rmc}(\tau_1),v_{i_\rmc}(\tau_1),x_j(\tau_1),v_j(\tau_1),n) \rmd n$, then update the velocity of particle $i_\rmc$ according to the rule
\[
v_{i_\rmc} \leftarrow v_{i_\rmc} + ((v_j - v_{i_\rmc})\cdot n)n\,;
\]
else jump to Step 4.

4. Update system time $t\leftarrow\tau_1$ and update the next collision time of particle $i_\rmc$ by $t_{i_\rmc}\leftarrow t_{i_\rmc} + \delta t$ with $\delta t \sim \mathrm{Exp}(\Lambda)$.

\State \textbf{end}
\State \textbf{Output:} Ensemble particles $\{(x_i(T),v_i(T))\}_{i=1}^N$.
\end{algorithmic}
\end{algorithm}

\subsection{Mean-field analysis for Boltzmann simulator}\label{sec:boltzmann_mean_field}

To provide the justification of these methods simulating the Boltzmann equation, one recognizes the algorithm calls for interactions between samples. The procedure naturally invites the theory of the mean-field derivation. We dedicate this section to provide a mean-field proof for the Nanbu and the Bird method.

It is clear that the coupling method of Section~\ref{sec:ensemble} can hardly be used. We switch our strategy and adopt the compactness argument for the associated martingale problems~\cite{GrMe:1997stochastic,Me:1996asymptotic,Sz:1984equations,ChDi:2022propagation_ii}. 
In Section~\ref{sec:particle_weak_form}, we lay out the weak formulation of the Vlasov-Boltzmann equation and introduce the associated martingale problem. 
In Section~\ref{sec:VB_martingale}, we follow a similar approach to introduce the weak formulation and the martingale problem for the two particle methods, the Nanbu sampler and the Bird sampler. 
A mean-field proof of the two particle methods is discussed in Section~\ref{sec:mean_field_proof}.

\subsubsection{Weak formulation and the Martingale problem}\label{sec:particle_weak_form}

Throughout the section, we assume that the PDE solution is absolutely continuous, so we rewrite $\mu_t(\rmd x, \rmd v)$ as $\mu_t(x,v)\rmd x \rmd v$. 
To obtain weak formulation, continuous test function $\phi\in C_\rmb^\infty(\Rb^{2d})$, we integrate it against~\eqref{eqn:mollified_boltzmann}, and obtain that
\begin{equation}\label{eqn:boltzmann_weak_int}
\frac{\rmd}{\rmd t} \int \phi(x,v) \mu_t(x,v) \rmd x \rmd v
= \int \Lc \phi(x,v) \mu_t(x,v) \rmd x \rmd v + \int \phi(x,v) \widetilde{Q} \mu_t(x,v) \rmd x \rmd v\,.
\end{equation}

The integral with respect to the collision term can be simplified through a pre-post-collisional change of variables~\cite{Vi:2002review}. According to~\eqref{eqn:new_velocities}, the change of velocities has a Jacobian whose determinant is $1$. Making use of~\eqref{eqn:velocity_diff} and that $\widetilde{q}$ depends only on $|v-w|$ and $(v-w)\cdot n$,
\[
\begin{aligned}
&\int \phi(x,v) \widetilde{Q} \mu_t(x,v) \rmd x \rmd v \\
&= \iint_{\Sb^{d-1}\times\Rb^{4d}} ~ ( \phi(x,v^\ast) - \phi(x,v) )\widetilde{q}(x,v,y,w,n) \mu_t(y,w) \mu_t(x,v) \, \rmd n \rmd y \rmd w \rmd x \rmd v \,.
\end{aligned}
\]
Let $\Kc$ denote the generator of the stochastic jump process
\begin{equation}
\Kc \phi(x,v,y,w) = \int_{\Sb^{d-1}} (\phi(x,v^\ast)-\phi(x,v)) \widetilde{q}(x,v,y,w,n) \rmd n  \,,
\end{equation}
and denote $\Cc[\mu_t]$ the collision operator acting on the test function $\phi$:
\begin{equation}
\Cc[\mu_t]\phi(x,v) = \left\langle \Kc \phi(x,v,\cdot,\cdot), \mu_t \right\rangle \\
= \iint_{\Rb^{2d}} \Kc \phi(x,v,y,w) \mu_t(y,w) \, \rmd y \rmd w \,.
\end{equation}
Then~\eqref{eqn:boltzmann_weak_int} has its weak formulation:
\begin{equation}\label{eqn:weak_mollified_boltzmann}
\frac{\rmd}{\rmd t} \left\langle \phi, \mu_t \right\rangle
 = \left\langle \Lc \phi, \mu_t \right\rangle
+ \left\langle \Cc[\mu_t]\phi, \mu_t \right\rangle = \left\langle \left(\Lc +\Cc[\mu_t]\right)\phi, \mu_t \right\rangle\,.
\end{equation}
Throughout the calculation, we use the bracket notation $\left\langle\cdot,\cdot\right\rangle$ to denote the integral of two functions in the full $(x,v)\in\Rb^{2d}$ domain.

This weak formulation~\eqref{eqn:weak_mollified_boltzmann} can then be translated to a martingale problem. For any $T>0$, we denote by $\Dc([0,T],\Rb^{2d})$ the Skorohod space\footnote{It is a convention to equip $\Dc([0,T],\Rb^{2d})$ with the $J_1$ topology~\cite{Bi:2013convergence}.} of all c\`adl\`ag functions defined over the interval $[0,T]$ and take values in $\Rb^{2d}$. Furthermore, denote $\Pc(\Dc([0,T],\Rb^{2d}))$ the collection of probability path measures over  $\Dc([0,T],\Rb^{2d})$. The mollified Vlasov-Boltzmann equation~\eqref{eqn:weak_mollified_boltzmann} can be translated to the following martingale problem:
\begin{problem}[Martingale problem associated with the mollified Vlasov-Boltzmann]\label{pro:mart_VB}
Find a path measure $P^\ast\in\Pc(\Dc([0,T],\Rb^{2d}))$ such that for any $\phi\in C_\rmb^\infty(\Rb^{2d})$ and for any initial condition $P_{\mathrm{ini}}\in\Pc_2(\Rb^{2d})$, the process defined by
\begin{equation}\label{eqn:boltzmann_martingale}
M^\phi_t(\omega) = \phi(Z_t(\omega)) - \phi(Z_0(\omega)) - \int_0^t \left(\Lc + \Cc[P^\ast_s]\right) \phi (Z_s(\omega)) \rmd s \,, \quad \forall\omega\in \Dc([0,T],\Rb^{2d}) \,,
\end{equation}
is a $P^\ast$-martingale and $(Z_0)_\sharp P^\ast = P_{\mathrm{ini}}$. Here we define $Z_t:\Dc([0,T],\Rb^{2d})\to\Rb^{2d}$ to be the canonical process such that $Z_t(\omega) = \omega(t)$ for any $\omega\in\Dc([0,T],\Rb^{2d})$, and $P^\ast_t = (Z_t)_\sharp P^\ast\in\Pc(\Rb^{2d})$ is the pushforward of the measure $P^\ast$ under the canonical map $Z_t$. In other words, $P^\ast_t$ is the $t$-marginal of the path measure.
\end{problem}

A weak form of the PDE is usually connected to the associated Martingale Problem, see e.g.~\cite{Ku:2011equivalence}. In our context, the solution to the Martingale Problem~\ref{pro:mart_VB} provides a solution to the weak form of the mollified Vlasov-Boltzmann equation~\eqref{eqn:weak_mollified_boltzmann}. In particular, let $P^\ast\in\Pc(\Dc([0,T],\Rb^{2d}))$ be the solution to the Martingale Problem~\ref{pro:mart_VB}, then one can check that its $t$-marginal satisfies the weak form equation of the mollified Vlasov-Boltzmann equation. To see this, we take the expectation on both sides of~\eqref{eqn:boltzmann_martingale}, and obtain
\begin{equation}\label{eqn:mart_to_boltzmann}
0 = \mathbb{E}_{P^\ast} [M_{t=0}^\phi] = \mathbb{E}_{P^\ast} [M_t^\phi]
= \mathbb{E}_{P^\ast} [\phi\circ Z_t] - \mathbb{E}_{P^\ast} [\phi\circ Z_0] - \int_0^t \mathbb{E}_{P^\ast} \left[ (\Lc \phi + \Cc[P^\ast_s] \phi ) \circ Z_s \right] \rmd s\,,
\end{equation}
where we used $\circ$ to denote function composition, the martingale property to propagate $\mathbb{E}_{P^\ast} [M_t^\phi]$ back to $\mathbb{E}_{P^\ast} [M_{t=0}^\phi]$, and that, according to the definition, $M_{t=0}^\phi(\omega)=0$.

Note that the $t$-marginal of $P^\ast$ is defined by $P^\ast_t = (Z_t)_\sharp P^\ast$, and hence for all $t$,
\[
\begin{aligned}
&\mathbb{E}_{P^\ast} [\phi\circ Z_t] = 
\int \phi(Z_t(\omega)) \rmd P^\ast (\omega)
= \left\langle \phi, P^\ast_t \right\rangle\,, \\ 
&\mathbb{E}_{P^\ast} \left[ (\Lc + \Cc[P^\ast_s])(\phi\circ Z_s) \right]
= \int ( \Lc \phi + \Cc[P^\ast_s] \phi ) (Z_s(\omega)) \rmd P^\ast (\omega)
= \left\langle \Lc \phi, P^\ast_s \right\rangle + \left\langle \Cc[P^\ast_s] \phi, P^\ast_s \right\rangle\,.
\end{aligned}
\]
Plug these back in~\eqref{eqn:mart_to_boltzmann} and identify $\mu_t = P^\ast_t$, we recover the weak form of the mollified Vlasov-Boltzmann equation~\eqref{eqn:weak_mollified_boltzmann}.

We do not delve into the details, but we do point out that just as one needs to show the well-posedness of the PDE, Problem~\ref{pro:mart_VB} also requires the well-posedness theory to ensure these is a unique path measure satisfying~\eqref{eqn:boltzmann_martingale}. Uniqueness can be proved under the mild regularity condition and the growth-condition of the cross-section kernel $\widetilde{q}(x,v,y,w,n)$. We refer interested readers to Theorem 2.2 in~\cite{Gr:1992nonlinear}.

\subsubsection{Weak formulation and the Martingale problems for the $N$-particle systems}\label{sec:VB_martingale}

We switch gears to study the discrete setting: the Nanbu method and the Bird method. In this section, we lay out the weak form PDEs, as Piecewise Deterministic Markov Processes for both methods and their associated martingale problems.

To do so, we first lift ourselves up first to the $N$-fold coordinate-system, with each coordinate keeping information of one particle: $x^N=(x_1,\dots,x_N)\in(\Rb^{d})^{\otimes N}$. Similarly, the velocity coordinate is also lifted up to the $N$-fold dimension by $v^N=(v_1,\dots,v_N)\in(\Rb^{d})^{\otimes N}$. We denote $\mu_t^N\in\mathcal{P}((\Rb^{2d})^{\otimes N})$ the distribution over the $N$-fold phase space.

Both Nanbu and Bird methods provide descriptions on how to propagate particles, so we can trace the evolution and translate them into the PDE language. For any given $\phi^N\in C_\rmb^\infty((\Rb^{2d})^{\otimes N})$,
\begin{equation}\label{eqn:particle_generator}
\frac{\rmd}{\rmd t} \left\langle \phi^N, \mu_t^N \right\rangle 
= \sum_{i=1}^N \left\langle \Lc_i \phi^N, \mu_t^N \right\rangle 
+ \sum_{i,j=1}^N \left\langle \Kc_{ij} \phi^N, \mu_t^N \right\rangle \,,
\end{equation}
where the free transport operator $\Lc_i$ acts only on the $i$-th coordinate pairs:
\begin{equation}\label{eqn:particle_free_generator}
    \Lc_i \phi^N(x^N,v^N) = v_i\cdot \nabla_{x_i} \phi^N(x^N,v^N) - \nabla_{x_i} f (x_i) \cdot \nabla_{v_i} \phi^N(x^N,v^N) \,, \quad i=1,\dots,N \,,
\end{equation}
and the form of the binary collision operator $\Kc_{ij}$ depends on the specific method in hand. For example, the Nanbu method only updates one coordinate, and it uses the following collision operator, for $i,j=1,\dots,N$:
\begin{equation}\label{eqn:particle_collision_generator}
    \Kc_{ij}^{(\rmN)} \phi^N(x^N,v^N) = \frac{1}{N}\int_{\Sb^{d-1}} (\phi^N(x^N,v^{N,i,\ast})-\phi^N(x^N,v^N)) \widetilde{q}(x_i,v_i,x_j,v_j,n) \rmd n\,,
\end{equation}
where
\[
v^{N,i,\ast} = (v_1,\dots,v_i + ((v_j-v_i)\cdot n)n,\dots,v_N)\,,
\]
is the post-collision velocity. We note that only the $i$-th coordinate is changed. The superindex $(\rmN)$ stands for Nanbu. The Bird method, on the other hand, picks a pair and updates them simultaneously. The associated collision operator then reads:
\begin{equation}
\Kc_{ij}^{(\rmB)} \phi^N(x^N,v^N) = \frac{1}{N}\int_{\Sb^{d-1}} (\phi^N(x^N,v^{N,i,j,\ast})-\phi^N(x^N,v^N)) \widetilde{q}(x_i,v_i,x_j,v_j,n) \rmd n\,,
\end{equation}
where
\[
v^{N,i,j,\ast} = (v_1,\dots,v_i + ((v_j-v_i)\cdot n)n,\dots,v_j + ((v_i-v_j)\cdot n)n, \dots, v_N)
\]
is the post-collision velocity, and both $(i,j)$-coordinates are changed. The super-index $(\rmB)$ stands for Bird. 

Similar to the Vlasov-Boltzmann equation, these weak formulations~\eqref{eqn:particle_generator} can be translated into martingale problems for the $N$-particle system. Similar to deriving the weak solution form, all terms are lifted up to the $N$-fold space. In particular, the $N$-fold path $\omega^{N}=(\omega_1,\dots,\omega_N)\in\Dc([0,T],(\Rb^{2d})^{\otimes N})$, and the path measure $P^N$ lives in the product space $\Pc(\Dc([0,T],(\Rb^{2d})^{\otimes N}))$, with the initial condition living in $P^N_{\mathrm{ini}}\in\Pc_2((\Rb^{2d})^{\otimes N})$. The martingale problem is to find a path on the product space:
\begin{problem}[Martingale problem associated with the $N$-particle system]\label{pro:mart_particle}
Find a path measure $P^N\in\Pc(\Dc([0,T],(\Rb^{2d})^{\otimes N}))$ such that for any test function $\phi^N\in C_\rmb((\Rb^{2d})^{\otimes N})$ and initial condition $P^N_{\mathrm{ini}}\in\Pc_2((\Rb^{2d})^{\otimes N})$, the process defined by
\begin{equation}\label{eqn:particle_martingale}
\begin{aligned}
&M^{\phi^N}_t(\omega^{N}) = \phi^N(Z^N_t(\omega^{N})) - \phi^N(Z^N_0(\omega^{N})) \\
&- \int_0^t \sum_{i=1}^N\Lc_i\phi^N (Z^N_s(\omega^{N})) + \sum_{i,j=1}^N \Kc_{ij} \phi^N (Z^N_s(\omega^{N})) \, \rmd s \,, \quad \forall\omega^{N}\in \Dc([0,T],(\Rb^{2d})^{\otimes N}) \,,
\end{aligned}
\end{equation}
is a $P^N$-martingale and $(Z^N_0)_\sharp P^N = P^N_{\mathrm{ini}}$. Here we define $Z^N_t:\Dc([0,T],(\Rb^{2d})^{\otimes N})\to(\Rb^{2d})^{\otimes N}$ to be the canonical process such that $Z^N_t(\omega^N) = \omega^N(t)$ for any $N$-fold path $\omega^{N}$.
\end{problem}

The well-posedness of Problem~\ref{pro:mart_particle} has also been investigated in~\cite{ShTa:1985central,EtKu:2009markov,JaSh:2013limit}.

\subsubsection{The mean-field analysis}\label{sec:mean_field_proof}

The mean-field analysis concerns connecting the two martingale problems and shows the convergence of empirical path measures. We now define the empirical path measure for the $N$-particle system 
\begin{equation}
Q_N(\omega^N) = \frac{1}{N} \sum_{i=1}^N \delta_{\omega_i}\in\Pc(\Dc([0,T],\Rb^{2d}))\,, \quad \forall \omega^N=(\omega_1,\dots,\omega_N)\in\Dc([0,T],(\Rb^{2d})^{\otimes N})\,.
\end{equation}
We note that $Q_N$ itself is a path measure, and due to randomness in $\omega^N$, $Q_N$ is not deterministic either, and it is a random path measure whose law lives in a probability measure space over the probability path measure. 
More specifically, $\Law(Q_N)\in\Pc(\Pc(\Dc([0,T],\Rb^{2d})))$. We want to prove that this random path measure, in the large $N$ limit, converges to a deterministic path measure. Hence its law is a delta measure in the probability over probability path measure space, centered at the solution that solves Problem~\ref{pro:mart_VB}.

\begin{theorem}[Theorem 4.5 in~\cite{GrMe:1997stochastic}]\label{thm:boltzmann_mean_field}
Suppose the initial condition $P_\mathrm{ini}^N$ is exchangeable in the sense that $P_\mathrm{ini}^N$ is invariant under the permutation of coordinates. 
Under Assumption~\ref{assumption:kernel_bound} and mild regularity assumptions on the cross-section $\widetilde{q}$ and the potential $f$, there is
\begin{equation}\label{eqn:path_measure_converge}
\Law(Q_N) := P^N \circ Q_N^{-1} \xrightarrow{N\to\infty} \delta_{P^\ast} \, \quad \text{ weak-$\ast$ in } \Pc(\Pc(\Dc([0,T],\Rb^{2d})))\,,
\end{equation}
where $P^\ast\in\Pc(\Dc([0,T],\Rb^{2d}))$ is the path measure solution to Problem~\ref{pro:mart_VB}.
\end{theorem}
\begin{remark}
The detailed assumption for this theorem to hold is technical. Generally speaking, according to~\cite{GrMe:1997stochastic}, these assumptions can be categorized into three groups: The gradient of $f$ is Lipschitz, the kernel $\widetilde{q}$ is smooth enough in its variables $(x,y,v,w)$, and the kernel $\widetilde{q}$ has suitable decay in its variable $(y,w)$. The same regularity assumptions are required in proving the well-posedness of the Martingale Problem~\eqref{pro:mart_VB} (see Theorem 2.2 in~\cite{Gr:1992nonlinear}). In the Bayesian sampling setting, the assumption on $f$ typically hold, and the assumptions on $\widetilde{q}$ can be satisfied through a proper algorithmic design.
\end{remark}

This convergence result leads to the mean-field limit. To be more precise, one can derive that the ensemble distribution living in $\Pc(\Rb^{2d})$, converges in law to the solution of the Vlasov-Boltzmann equation.
More precisely:
\begin{corollary}\label{lem:boltzmann_ensemble}
Consider the $N$-particle ensemble distribution, defined by
\begin{equation}
Q_{N,t}(\omega^N) := \frac{1}{N} \sum_{i=1}^N \delta_{(x_i(t),v_i(t))} \in \Pc(\Rb^{2d}) \,,
\end{equation}
where $\{(x_i(t),v_i(t))\}_{0\leq t\leq T}\in\Dc([0,T],\Rb^{2d})$ denotes the path of the $i$-th particle. Then $Q_{N,t}$ is a random measure on $\Rb^{2d}$.
Under the same assumption as Theorem~\ref{thm:boltzmann_mean_field}, and taking further assumptions that $Z_t:\Dc([0,T],\Rb^{2d})\to\Rb^{2d}$ is a continuous map, we know that $Q_{N,t}$ converges in distribution to $P^\ast_t$, which we identify with $\mu_t$, the solution of the mollified Vlasov-Boltzmann equation~\eqref{eqn:mollified_boltzmann}. In other words, we have the convergence of the law of $Q_{N,t}$:
\begin{equation}\label{eqn:boltzmann_ensemble_convergence}
\mathrm{Law}(Q_{N,t}) := P^N\circ Q^{-1}_{N,t} \xrightarrow{N\to\infty} \delta_{\mu_t} \quad \text{ weak-$\ast$ in } \Pc(\Pc(\Rb^{2d})) \,.
\end{equation}
\end{corollary}
\begin{proof}
Note that the $N$-particle ensemble distribution at time $t$ can be written as the pushforward of the empirical path measure:
\begin{equation}\label{eqn:path_emp_to_particle_emp}
Q_{N,t}(\omega^N) = (Z_t)_\sharp Q_N(\omega^N) \in \Pc(\Rb^{2d}) \,.
\end{equation}
For any test function $g\in C_\rmb(\Pc(\Rb^{2d}))$, by making use of~\eqref{eqn:path_emp_to_particle_emp} and the pushforward definition:
\begin{equation}\label{eqn:test_particle_emp}
\int_{\Pc(\Rb^{2d})} g(\mu) \, \rmd P^N \circ Q^{-1}_{N,t}(\mu)
= \int_{\Pc(\Dc([0,T],\Rb^{2d}))} g((Z_t)_\sharp P) \, \rmd P^N \circ Q^{-1}_{N}(P) \,.
\end{equation}

Assuming the continuity of the canonical map $Z_t$, we know that the functional $h(P) = g((Z_t)_\sharp P)$ is bounded and continuous, that is, $h\in C_\rmb(\Pc(\Dc([0,T],\Rb^{2d})))$.

Noting the measure $P^\ast_t$ is a pushforward of $P^\ast$: $P^\ast_t = (Z_t)_\sharp P^\ast \in \Pc(\Rb^{2d})$, we draw the conclusion from Theorem~\ref{thm:boltzmann_mean_field} to rewrite~\eqref{eqn:test_particle_emp} for the convergence of the law of the ensemble distribution:
\begin{equation}\label{eqn:ensemble_functional_conv}
\int_{\Pc(\Rb^{2d})} g(\mu) \, \rmd P^N \circ Q^{-1}_{N,t}(\mu)
\xrightarrow{N\to\infty}  \int_{\Pc(\Dc([0,T],\Rb^{2d}))} g((Z_t)_\sharp P) \, \rmd \delta_{P^\ast}(P) 
= \int_{\Pc(\Rb^{2d})} g(\mu) \, \rmd \delta_{\mu_t}(\mu)\,.
\end{equation}
\end{proof}

In general, the map $Z_t$ is not a continuous map from $\Dc([0,T],\Rb^{2d})$ to $\Rb^{2d}$, but the conclusion still holds. 
We refer the readers to~\cite{GrMe:1997stochastic,Me:1996asymptotic} for a rigorous justification of the convergence.

\begin{remark}
    We stress a few remarks.
    \begin{itemize}
        \item There are many layers of randomness. The algorithm encodes Poisson randomness, so even with fixed initial data, each run of the algorithm provides a different configuration of $Q_{N,t}$, and thus it presents a probability measure over $\Rb^{2d}$, which means $Q_{N,t}\in\Pc(\Rb^{2d})$. At the initial time, $N$ samples are drawn at random, and putting all $Q_{N,t}$ equipped with different initial data together, one forms another probability measure over $\Pc(\Rb^{2d})$. Theorem~\ref{thm:boltzmann_mean_field}, and consequently Corollary~\ref{lem:boltzmann_ensemble} claim the convergence of the law of $Q_{N,t}$, in the space of probability over probability on $\Rb^{2d}$, to a Dirac delta centered on the solution to the Boltzmann equation~\eqref{eqn:mollified_boltzmann}.
        \item As spelled out in~\eqref{eqn:ensemble_functional_conv}, the convergence is rather weak: The weak-$\ast$ convergence on $\Pc(\Pc(\Rb^{2d}))$ means that the test function $g$ in actually a \emph{functional} over $\Pc(\Rb^{2d})$ itself. However, this weak convergence is strong enough to provide convergence of all moments. To see this, we set
        \[
        g(\mu) = \int (|x|^p+|v|^p)\mu(\rmd x\rmd v)\,,
        \]
        a moment-taking functional. It is clearly a continuous bounded functional over $\Pc_p(\Rb^{2d})$. In the limit $N\to\infty$, $\int_{\Pc(\Rb^{2d})} g(\mu) \, \rmd \delta_{\mu_t}(\mu)=g(\mu_t)$ provides the moments of the solutions to the Vlasov-Boltzmann equation. The left-hand side in~\eqref{eqn:boltzmann_ensemble_convergence} when tested by moments, or equivalently~\eqref{eqn:ensemble_functional_conv}, can be interpreted as
        \[       \Eb_{P^N_{\mathrm{ini}}}\Eb_{\omega}\left(\frac{1}{N}\sum_i(|x_i(t)|^p+|v_i(t)|^p)\right)
        \]
        with the two expectations taking over different initial configurations and different Poisson clock ringings, and hence all randomness involved in the algorithms.
        \item Although we do not dive into details, the proof of Theorem~\ref{thm:boltzmann_mean_field} follows a fixed set of machinery, see~\cite{GrMe:1997stochastic,Me:1996asymptotic}. The proof strategy heavily relies on the compactness argument and is composed of two steps.
\begin{itemize}
    \item[Step 1.] The law of $Q_N$ is tight so that there exist a limit point $\mathfrak{Q}_\infty\in\Pc(\Pc(\Dc([0,T],\Rb^{2d})))$;
    \item[Step 2.] For $\mathfrak{Q}_\infty$-almost surely $P\in\Pc(\Dc([0,T],\Rb^{2d}))$, $P$ solves the mollified Vlasov-Boltzmann Martingale Problem. 
\end{itemize}
These two steps combined, and calling the uniqueness to the Martingale Problem~\ref{pro:mart_VB}, we have $\mathfrak{Q}^\infty = \delta_{P^\ast}$ with $P^\ast$ providing the weak solution to the mollified Vlasov-Boltzmann equation~\eqref{eqn:weak_mollified_boltzmann}.
    \end{itemize}
\end{remark}

\subsection{Numerical experiments}\label{sec:boltzmann_numerics}

In this section, we report on numerical experiments. We consider problems with $d = 1$ and $d = 2$ in the following. 
It is worth noting that for $d=1$, the collision operator in the Vlasov-Boltzmann equation~\eqref{eqn:boltzmann} becomes zero, whereas the mollified collision operator remains nontrivial and continues to play a role in the convergence to the target distribution.

To set up the problem, we consider a target probability measure $\rho^\ast\in\Pc(\Rb^d)$, where $\rho^\ast\propto\exp(-f(x))$. 
Initially, we sample $N$ particles $\{(x_i(0),v_i(0))\}_{i=1}^N$ independently from a distribution $\mu_0\in\Pc(\Rb^{2d})$, defined by
\begin{equation}\label{eqn:numerics_initial_distribution}
\mu_0\propto \mathds{1}_{|x|_\infty\leq L} \cdot \exp\left(-\frac{|v|^2}{2\sigma^2}\right) \,, \quad \forall (x,v)\in\Rb^{2d} \,.
\end{equation}
This initial condition is chosen to fit the energy condition~\eqref{eqn:boltzmann_convergence_conditions} required for achieving the convergence of the PDE, Theorem~\ref{thm:boltzmann_equilibrium}. This translates to the following configuration for $\sigma$ when $L$ is fixed.

\[
\sigma^2 = 1 + \frac{2}{d} \left(\frac{\int_{\Rb^d} f(x) \exp(-f(x)) \rmd x}{ \int_{\Rb^d} \exp(-f(x)) \rmd x} - \frac{1}{(2L)^d} \int_{|x|_\infty\leq L} f(x) \rmd x \right) \,.
\]
We take the following regularized collision kernel
\begin{equation}\label{eqn:boltzmann_kernel_numerics}
\widetilde{q}(x,v,y,w,n) = \frac{1}{\left(\epsilon\sqrt{\pi}\right)^d}\exp\left(-\frac{|x-y|^2}{\epsilon^2}\right) \,,
\end{equation}
so that the total cross-section $\Lambda = \frac{2} {\epsilon\sqrt{\pi}}$ if $d=1$ and $\Lambda = \frac{2} {\epsilon^2\pi}$ if $d=2$. Since the regularized collision kernel is independent of $n$, upon collision, the vector $n$ is sampled uniformly in the set $\{-1,1\}$ when $d=1$, and uniformly in the circle $\Sb^1$ when $d=2$.
\begin{remark}
We choose the form of $\widetilde{q}$ for its simplicity. We do observe that the numerical performance of the algorithm depends on $\epsilon$. Preliminary numerical test suggests larger $\epsilon$ leads to a faster convergence of the PDE, and requires a smaller number of particles to simulate the PDE. This observation suggests a better algorithm performance for a larger $\epsilon$. We do not yet have a theoretical justification to support this observation. More numerical examples are also needed to support this finding.
\end{remark}

To measure the sampling performance, we consider two different metrics using the Kullback–Leibler (KL) divergence: the KL-divergence between the $x$-marginal distribution, defined by
\begin{equation}
KL_X(\mu_X(t) || \mu_X^\ast) = \int_{\Rb^d} \mu_X(t,x) \ln\frac{\mu_X(t,x)}{\mu_X^\ast(t,x)} \rmd x = \int_{\Rb^d} \mu_X(t,x) \ln \mu_X(t,x) \rmd x + \int_{\Rb^d} f(x) \mu_X(t,x) \rmd x \,,
\end{equation}
where $\mu_X$ denotes the $x$-marginal of $\mu$, and the KL divergence between the phase space distribution, defined by
\begin{equation}
\begin{aligned}
KL(\mu(t) || \mu^\ast) &= \int_{\Rb^{2d}} \mu(t,x,v) \ln\frac{\mu(t,x,v)} {\mu^\ast(t,x,v)} \, \rmd x \, \rmd v \\
&= \int_{\Rb^{2d}} \mu(t,x,v) \ln \mu(t,x,v) \, \rmd x \, \rmd v + \int_{\Rb^{2d}} \left( f(x) + \frac{v^2}{2} \right) \mu(t,x,v) \, \rmd x \, \rmd v \,.
\end{aligned}
\end{equation}
In both the Nanbu sampler and the Bird sampler, the measure $\mu(t,x,v)$ is approximated by the ensemble distribution, defined by $\mu(t,x,v) = \frac{1}{N} \sum_{i=1}^N \delta_{(x_i(t),v_i(t))}(x,v)$. 
Neither of the KL-divergence is defined for the ensemble distribution. 
To circumvent this difficulty, we evaluate the following mollified KL-divergence, defined by
\begin{equation}\label{eqn:KL_numerics}
\begin{aligned}
    KL_X^\delta &= \frac{1}{N} \sum_{i=1}^N \ln\left( \frac{1}{N} \sum_{j=1}^N \rho^\delta(x_i(t)-x_j(t)) \right) + \frac{1}{N} \sum_{i=1}^N f(x_i(t)) \,. \\
    KL^\delta &= \frac{1}{N} \sum_{i=1}^N \ln\left( \frac{1}{N} \sum_{j=1}^N \rho^\delta(x_i(t)-x_j(t)) \rho^\delta(v_i(t)-v_j(t)) \right)   + \frac{1}{N} \sum_{i=1}^N \left[ f(x_i(t)) + \frac{v_i^2(t)}{2} \right]\,.
\end{aligned}
\end{equation}
Here, we take the mollifier $\rho^\delta(x) = (8\pi\delta^2)^{-d/2} \exp(-|x|^2/(8\delta^2))$.

In all the examples below, to generate reference solutions, we use the Inverse Transform Sampling (see, e.g., page 102 in~\cite{Ge:2003random}) to obtain $N=1000$ samples from $\mu^\ast_X \propto e^{-f(x)}$. In the simple settings like ours ($d=1,2$), Inverse Transform Sampling can still be deployed relatively easily~\cite{Ge:2003random}. The obtained samples are then used to compute a baseline value for the regularized KL divergences as the reference data (we term it the baseline solution in the plots).

Both algorithms are run with $N=1000$ particles. For simulating the free transport term~\eqref{eqn:hamiltonian}, Verlet algorithm~\cite{SwAnBeWi:1982computer} is deployed as the discretization step. In~\eqref{eqn:boltzmann_kernel_numerics}, we take the parameter $\epsilon = 1$ when $d=1$ and $\epsilon = 4$ when $d=2$ to ensure the time interval between collision is not too large. In~\eqref{eqn:KL_numerics}, we take the regularization constant $\delta = 0.3$ across all examples. 
We take the support of initial condition $L = 2$ in~\eqref{eqn:numerics_initial_distribution}.

In the first example, we set $f(x) = \frac{x^2}{2}$. Figure~\ref{fig:boltzmann_ex1} shows the decrease in the relative entropy $KL_X^\delta$, $KL^\delta$ and the histogram of samples obtained by the Nanbu and Bird sampler. Both methods provide ensemble distribution that quickly converge to the target distribution. 
The phase space relative entropy decays monotonically at the initial stage, which is aligned with the $H$-theorem~\cite{Vi:2002review}. 
This is because the phase space entropy can be decomposed into the sum of the absolute entropy and the total energy, where the former monotonically decreases, and the latter remains constant.
Following the initial stage, the phase space entropy oscillates near the equilibrium due to the randomness of the particles.
We observe that the Nanbu method exhibits larger fluctuations and reaches saturation earlier. This is because the energy is not conserved upon collisions in the Nanbu method. 
In comparison, the Bird method demonstrates better stability due to the energy conservation.
\begin{figure}[htbp]
  \centering
  \includegraphics[width=0.22\textwidth]{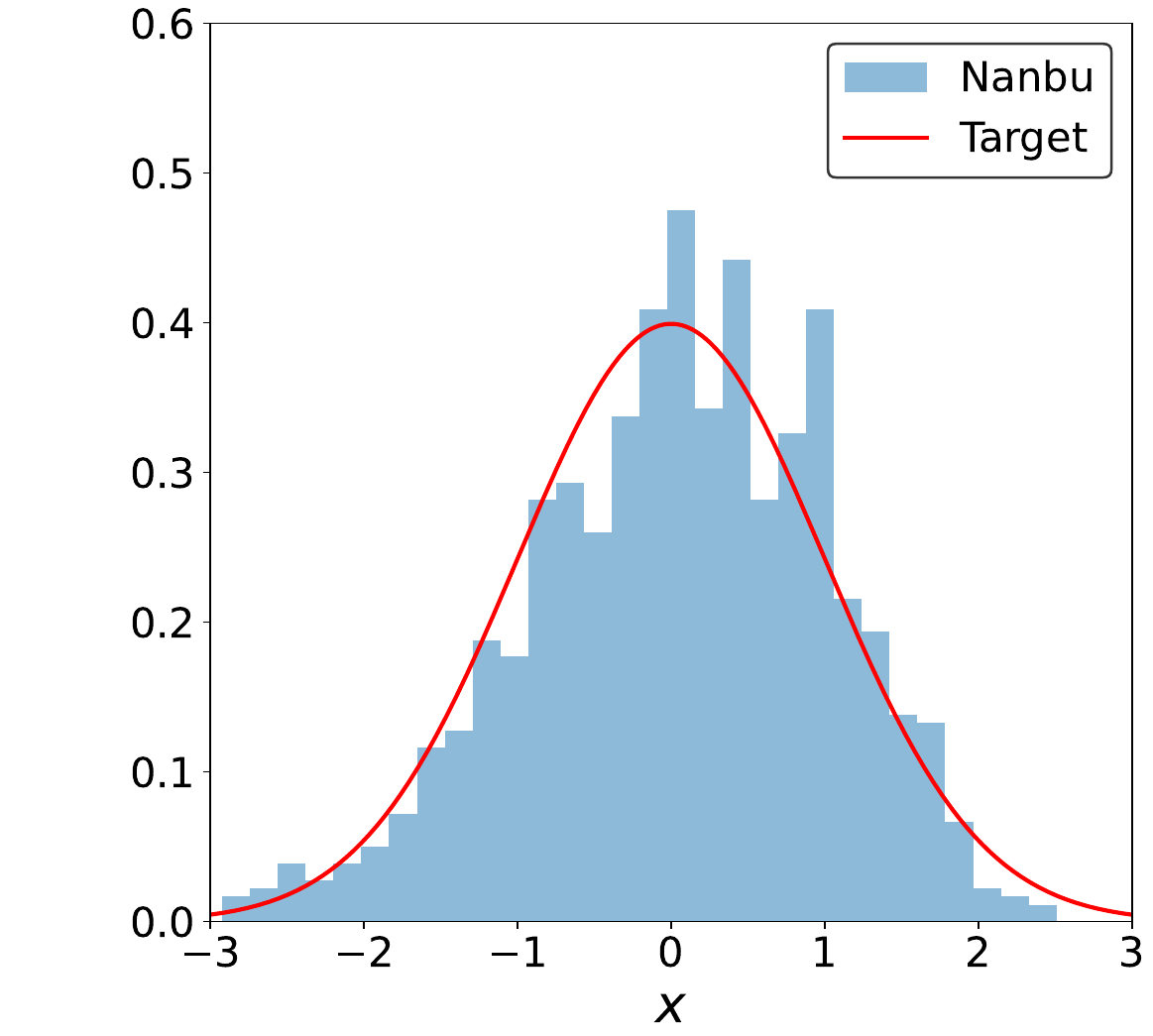}
  \includegraphics[width=0.22\textwidth]{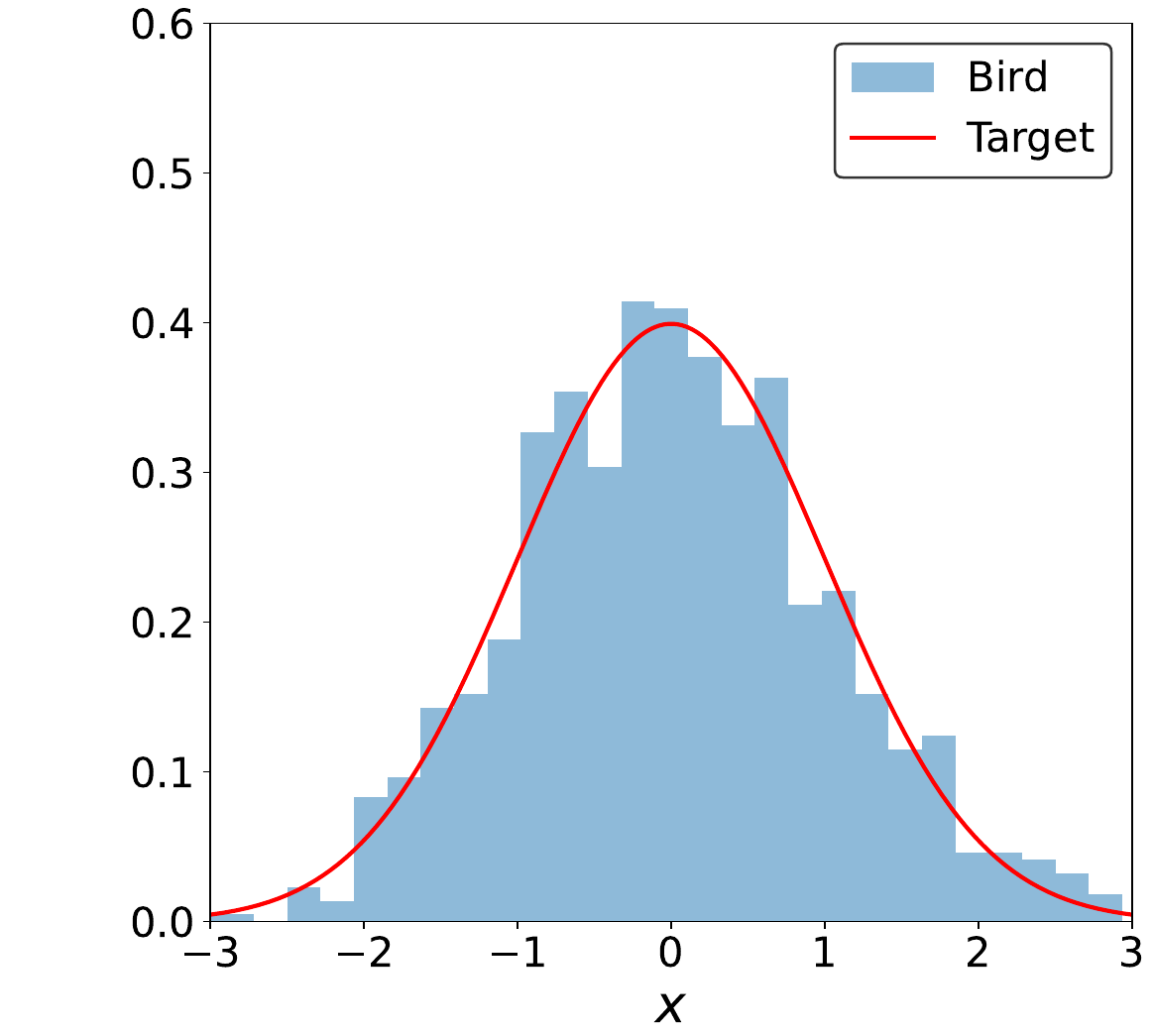}
  \includegraphics[width=0.22\textwidth]{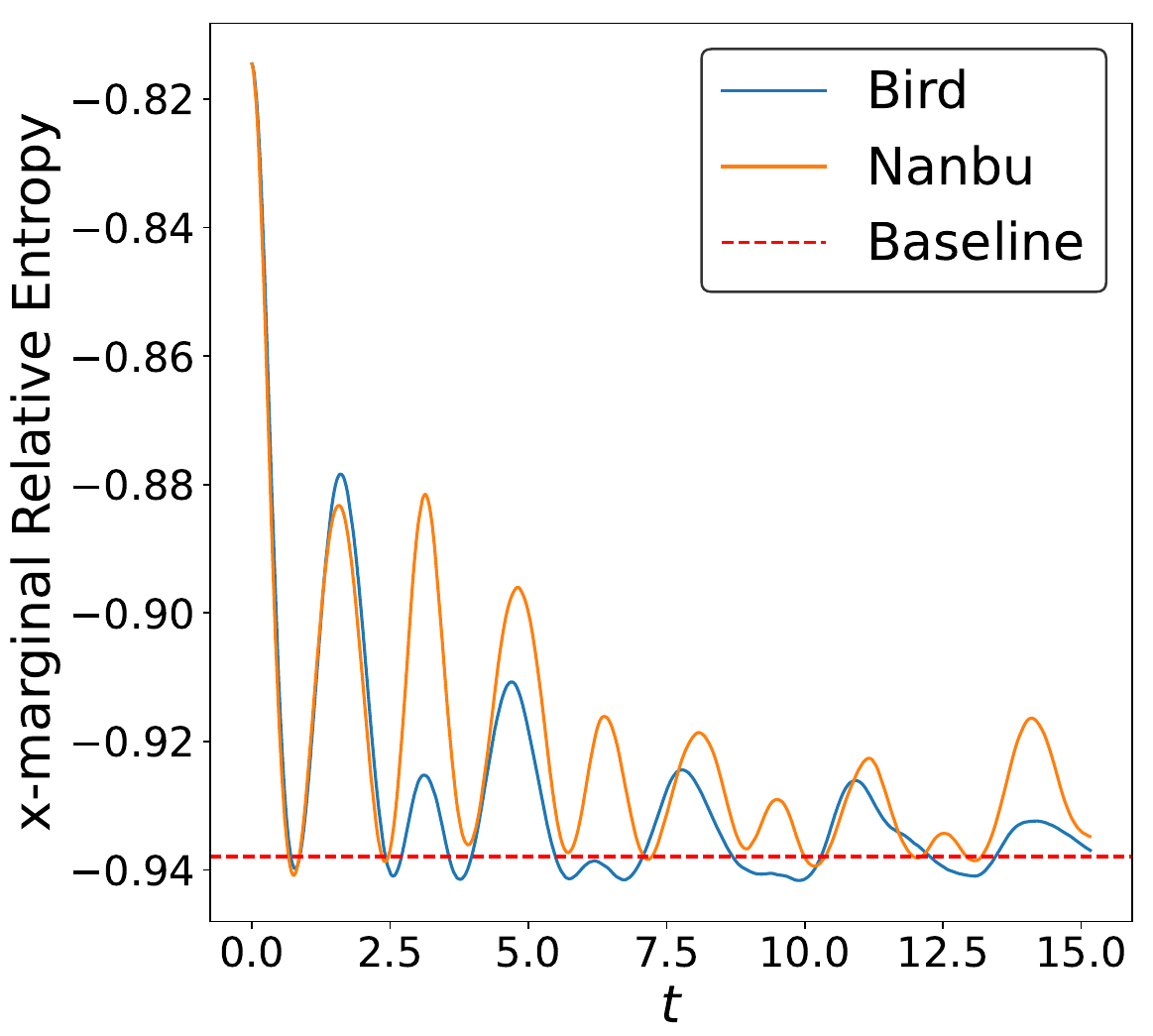}
  \includegraphics[width=0.22\textwidth]{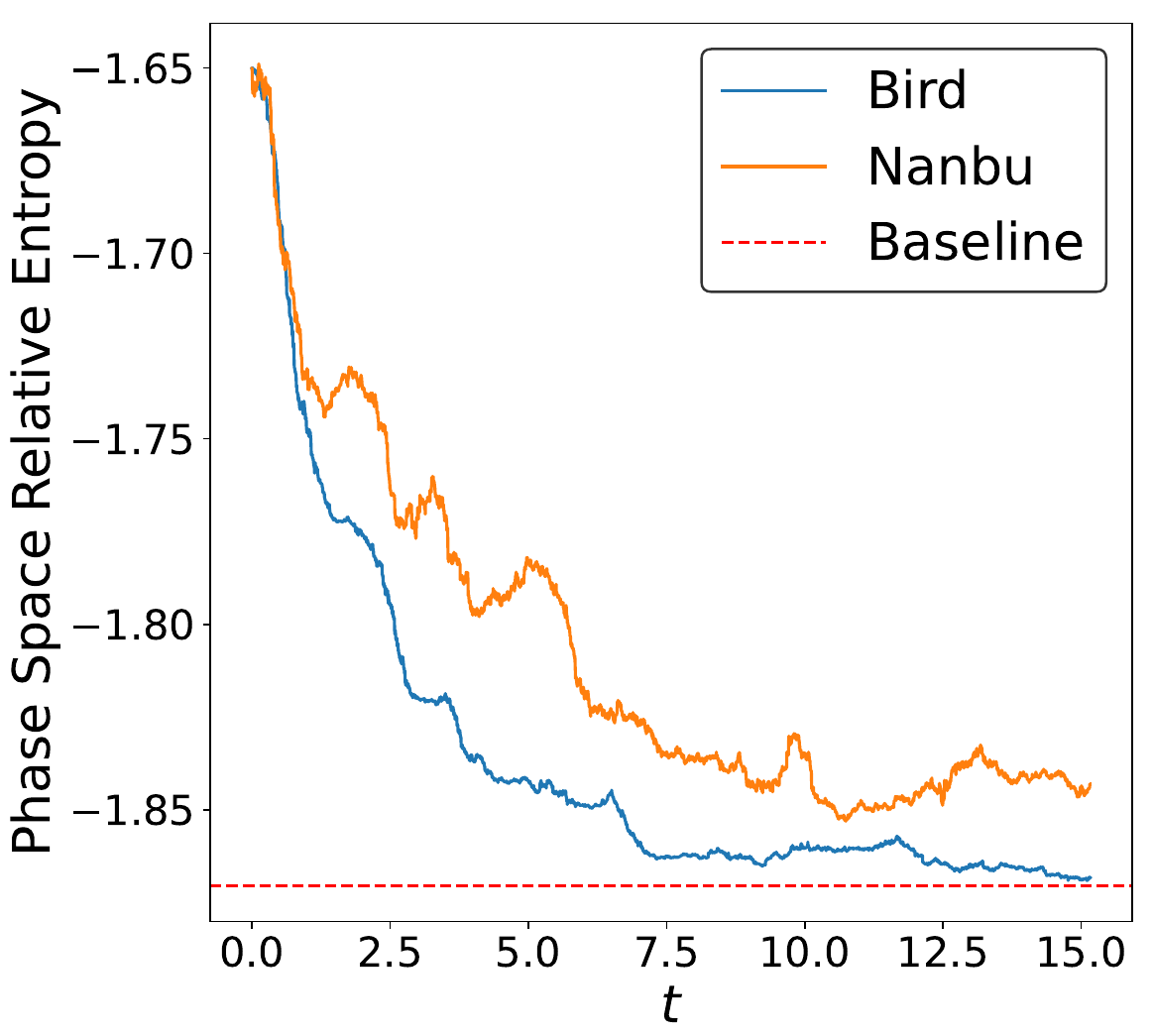}
  
  \caption{\textbf{Left Two Plots:} The target distribution $\rho^\ast \propto \exp(-\frac{x^2}{2})$ and the histogram of samples obtained from the Nanbu sampler (Leftmost) and the Bird sampler (Second-Left). \textbf{Right Two Plots:} The regularized relative entropy $KL_X^\delta$ (Second-Right) and $KL^\delta$ (Rightmost) of the Bird sampler and the Nanbu sampler. The baseline value is computed by using $N=1000$ samples obtained by the Inverse Transform method.
  }
  \label{fig:boltzmann_ex1}
\end{figure}

In the second example, we set $f(x) = (x-1)^2(x+1)^2$ so the target distribution is clearly not log-concave. In Figure~\ref{fig:boltzmann_ex2} we nevertheless show the relative entropy $KL_X^\delta$ and $KL^\delta$ decreases over time. This is a strong indication in this strongly non-convex case, the Boltzmann simulator can nevertheless achieve global convergence.
\begin{figure}[htbp]
  \centering
  \includegraphics[width=0.22\textwidth]{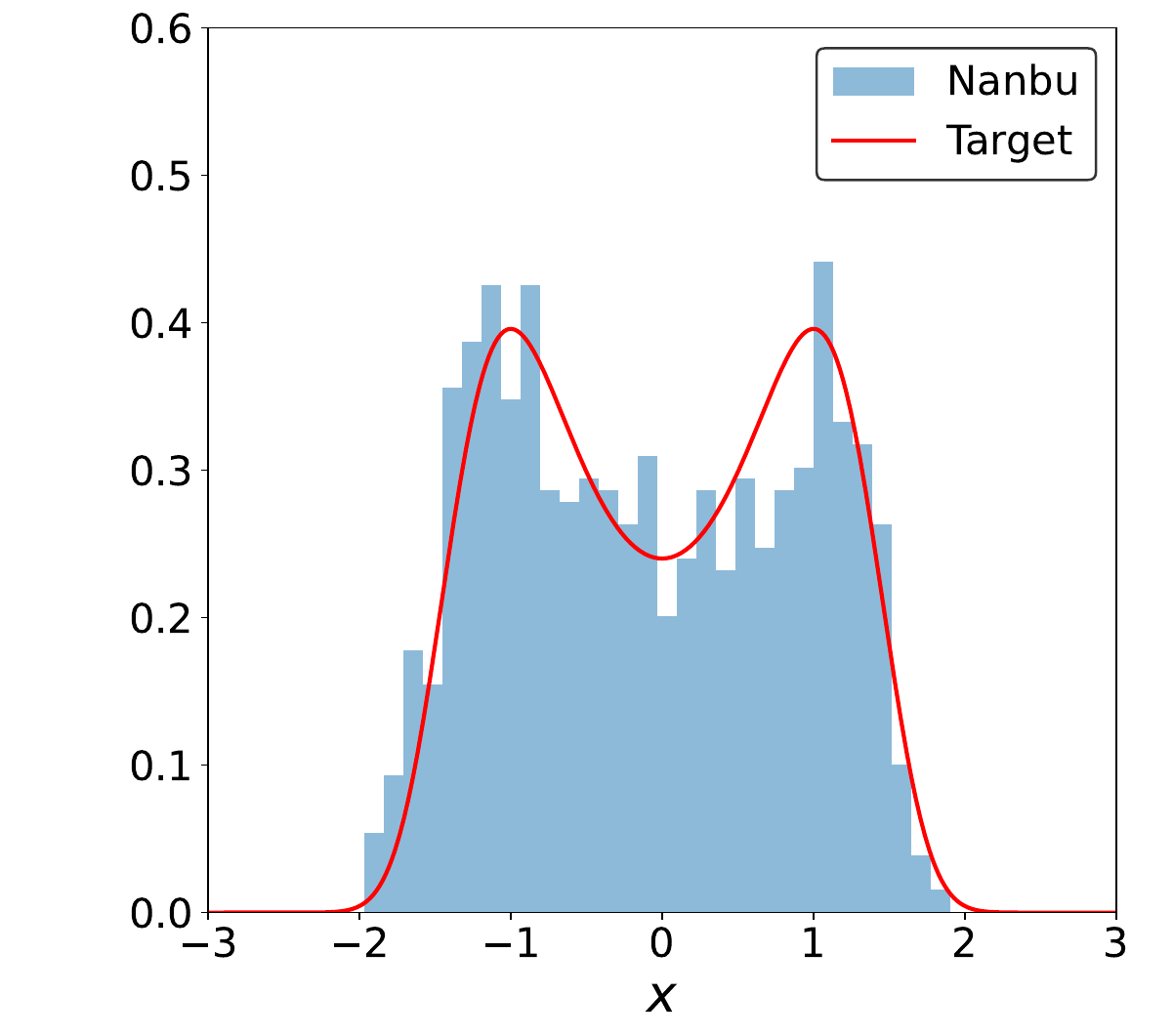}
  \includegraphics[width=0.22\textwidth]{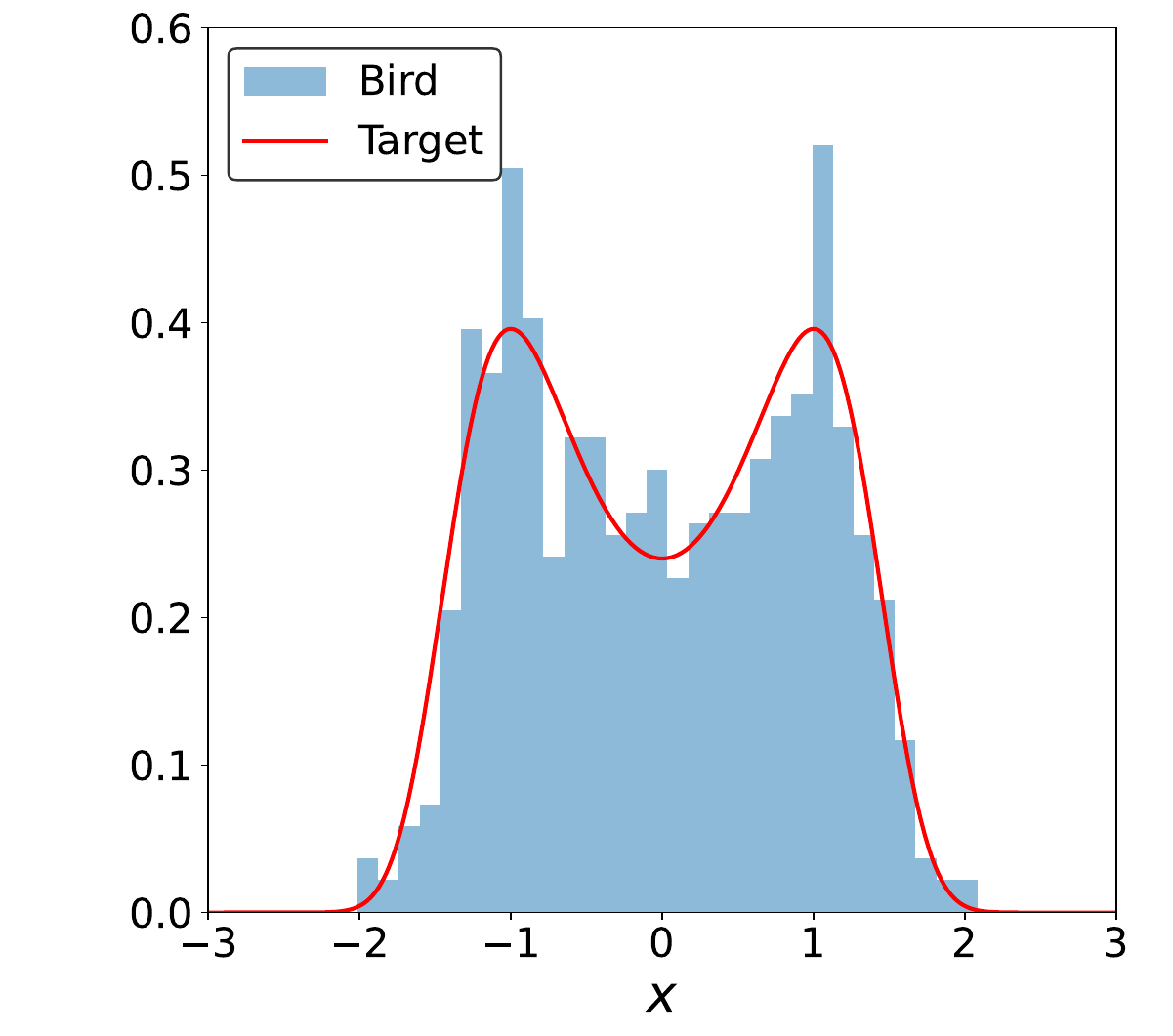}
  \includegraphics[width=0.22\textwidth]{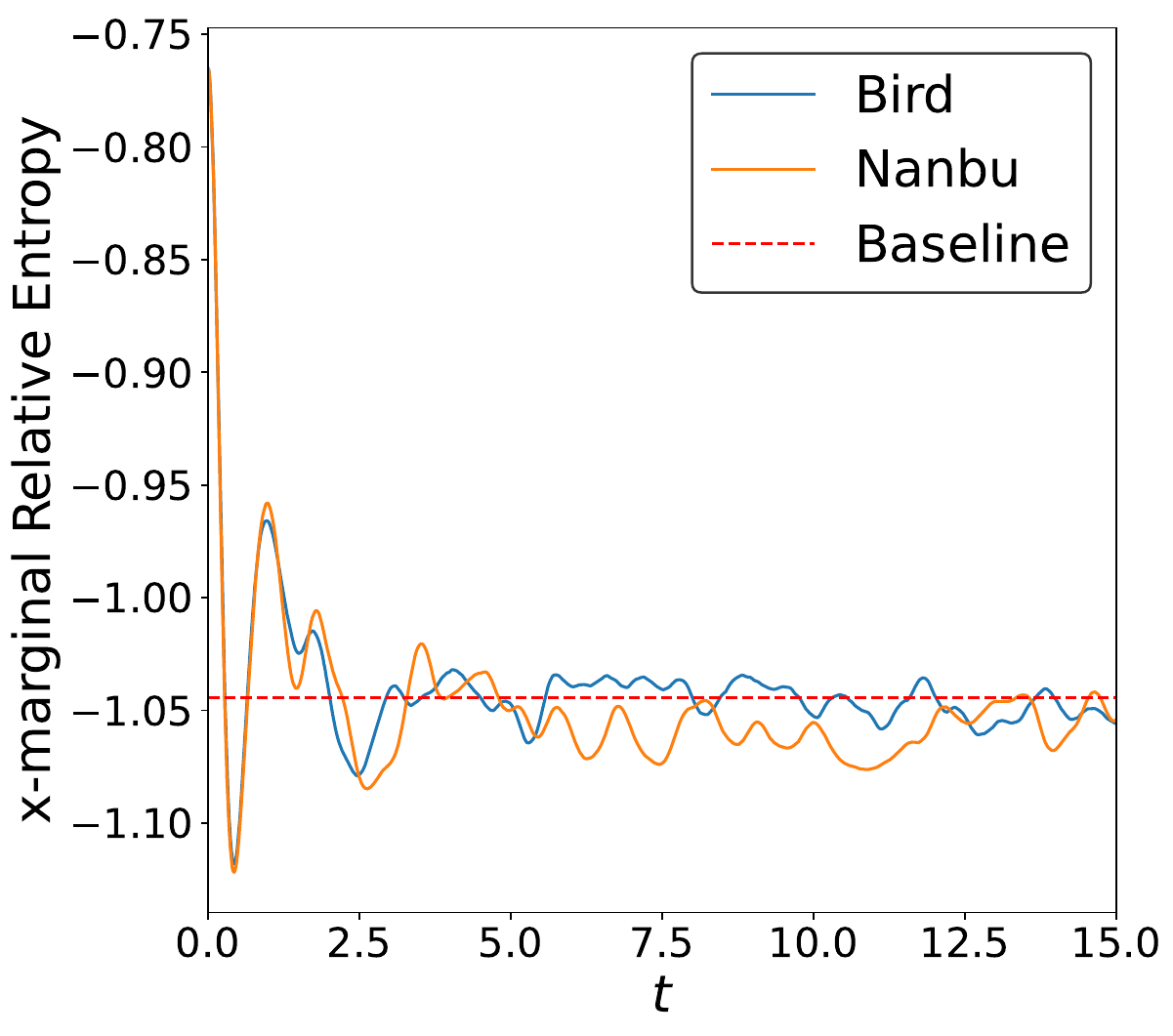}
  \includegraphics[width=0.22\textwidth]{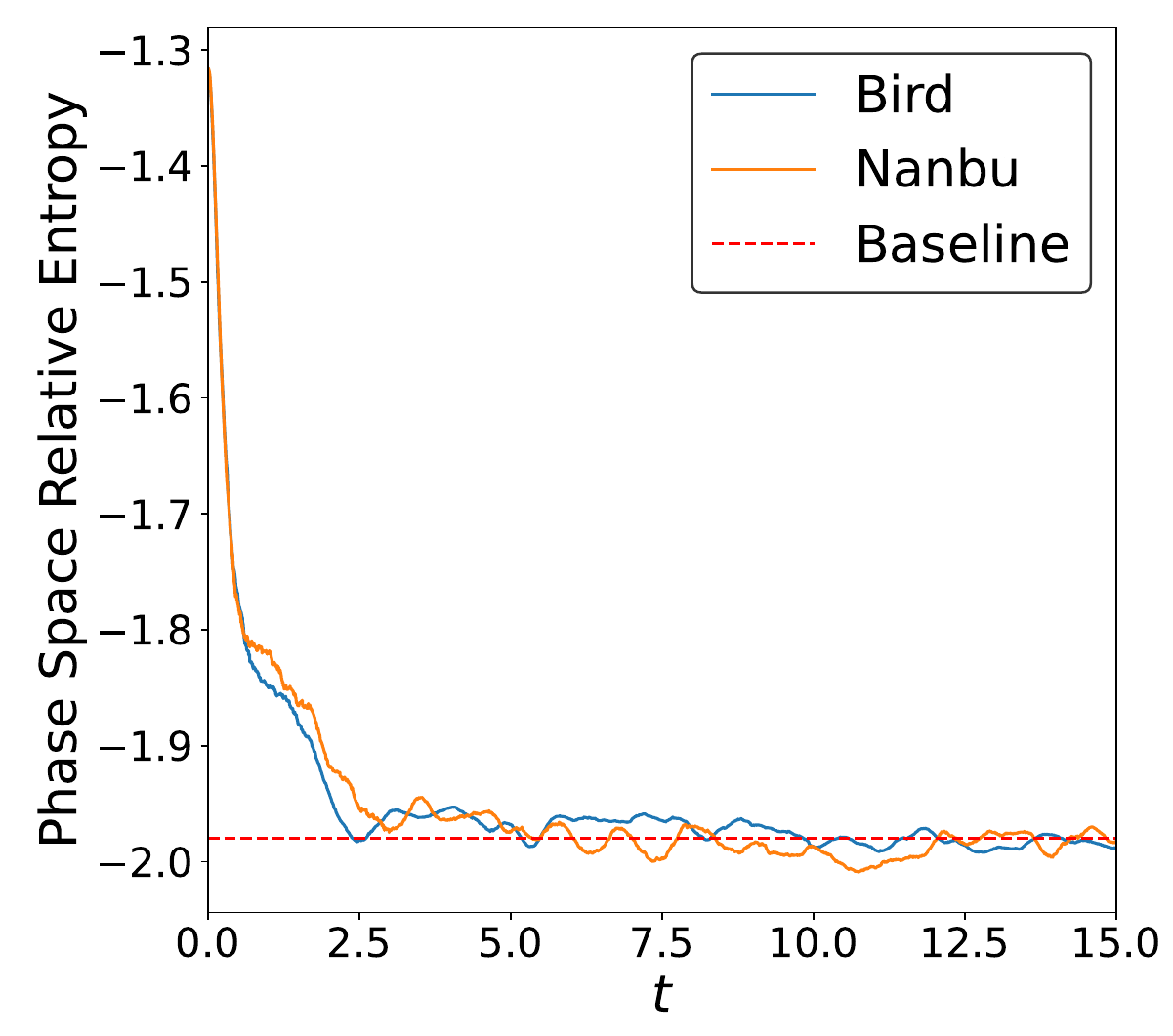}
  \caption{\textbf{Left Two Plots:} The target distribution $\rho^\ast \propto \exp(-(x-1)^2(x+1)^2)$ and the histogram of samples obtained from the Nanbu sampler (Leftmost) and the Bird sampler (Second-Left). \textbf{Right Two Plots:}  The regularized relative entropy $KL_X^\delta$ (Second-Right) and $KL^\delta$ (Rightmost) of the Bird sampler and the Nanbu sampler. The baseline value is computed by using $N=1000$ samples obtained by the Inverse Transform method.
  }
  \label{fig:boltzmann_ex2}
\end{figure}

In the third and fourth example, we consider two-dimensional problems with potential $f(x,y) = 0.5(x^2+y^2)$ and $f(x,y) = 0.1((x-1)^2+(y-1)^2)((x+1)^2+(y+1)^2)$. 
While the former one corresponds to the Gaussian distribution, the latter is a double-well potential and the target distribution is not log-concave.
In Figure~\ref{fig:boltzmann_ex3} and Figure~\ref{fig:boltzmann_ex4}, both the relative entropy $KL_X^\delta$ and $KL^\delta$ decreases over time, regardless of the convexity of the potential. 
\begin{figure}[htbp]
  \centering
  \includegraphics[width=0.22\textwidth]{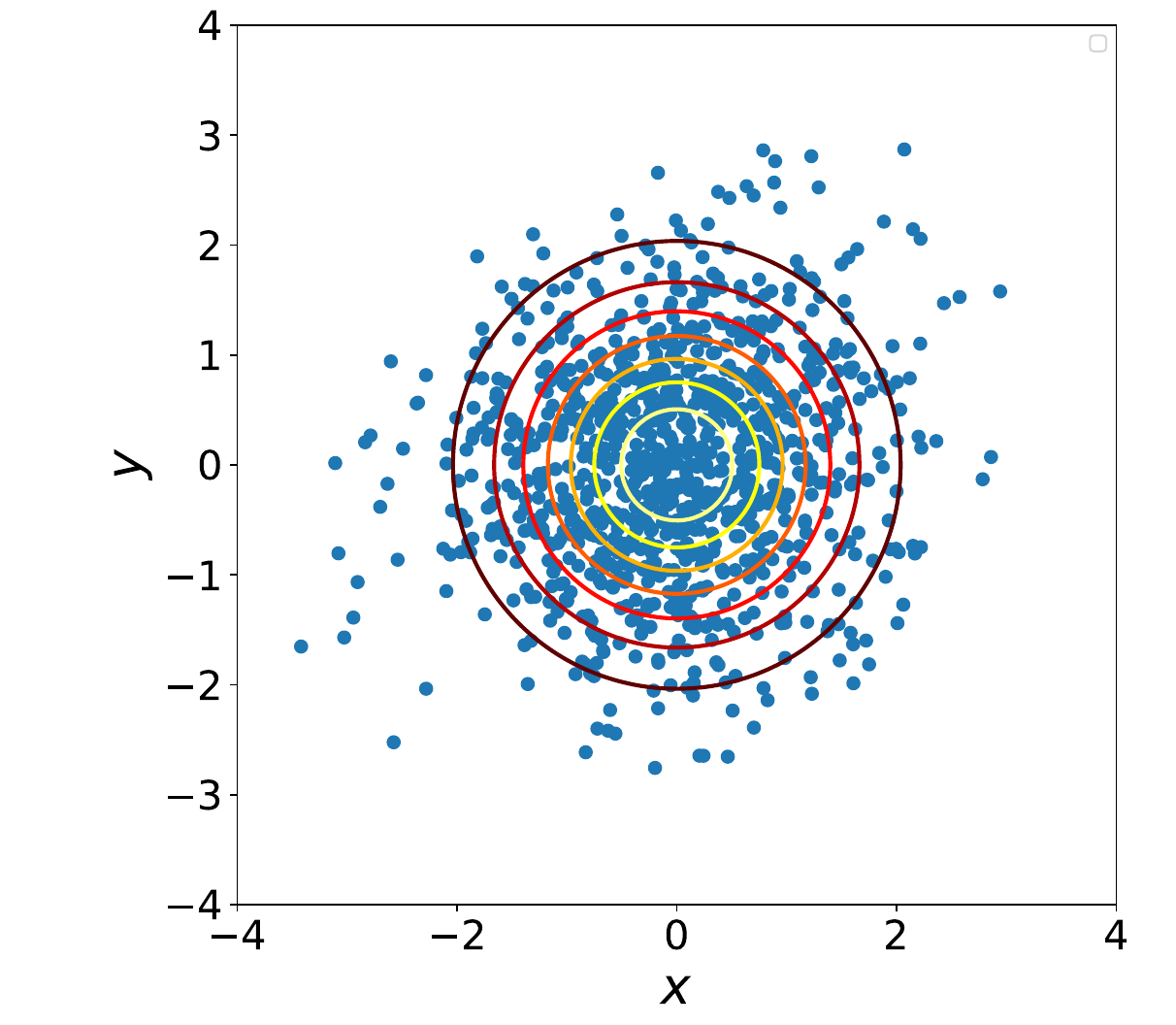}
  \includegraphics[width=0.22\textwidth]{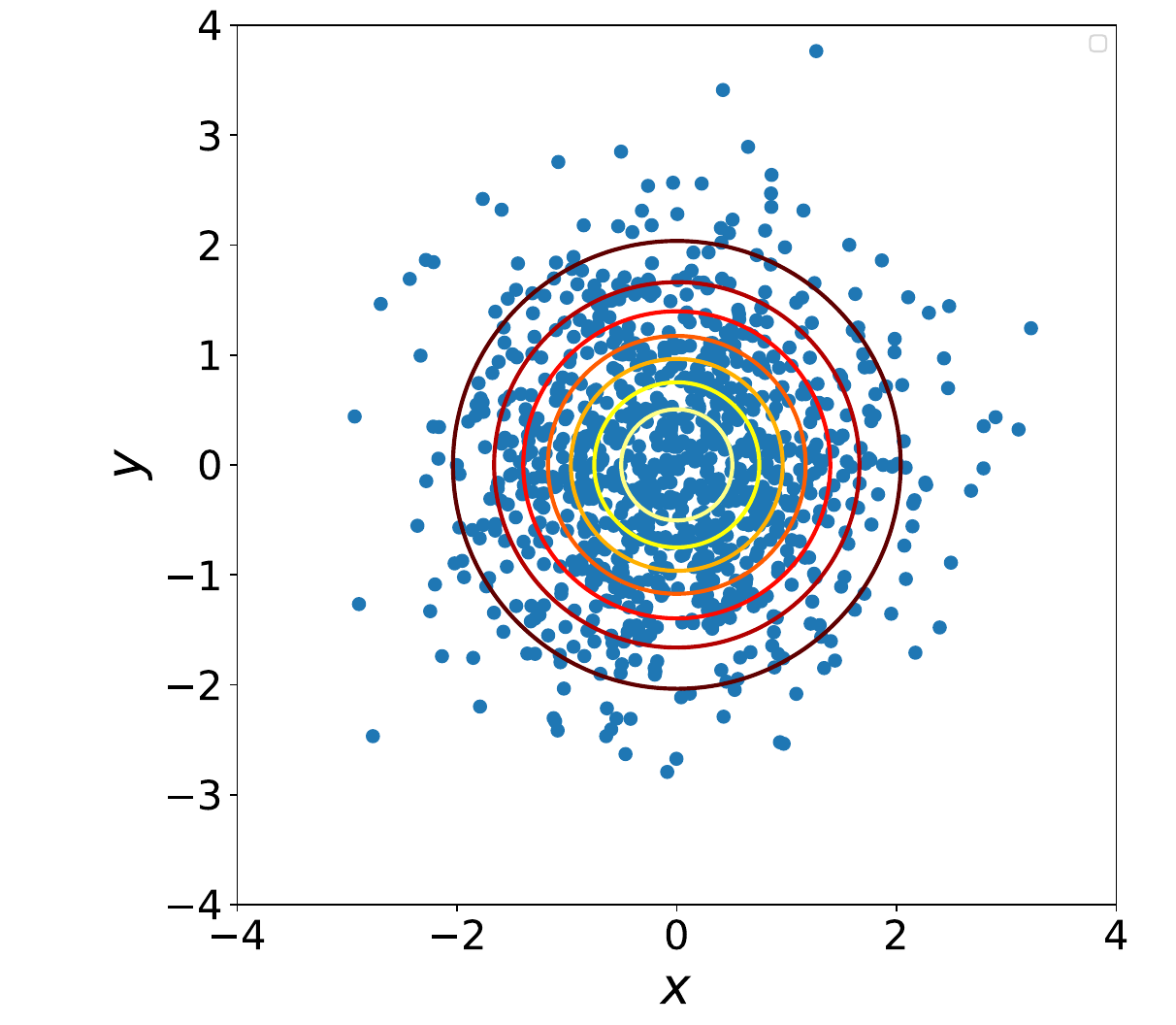}
  \includegraphics[width=0.22\textwidth]{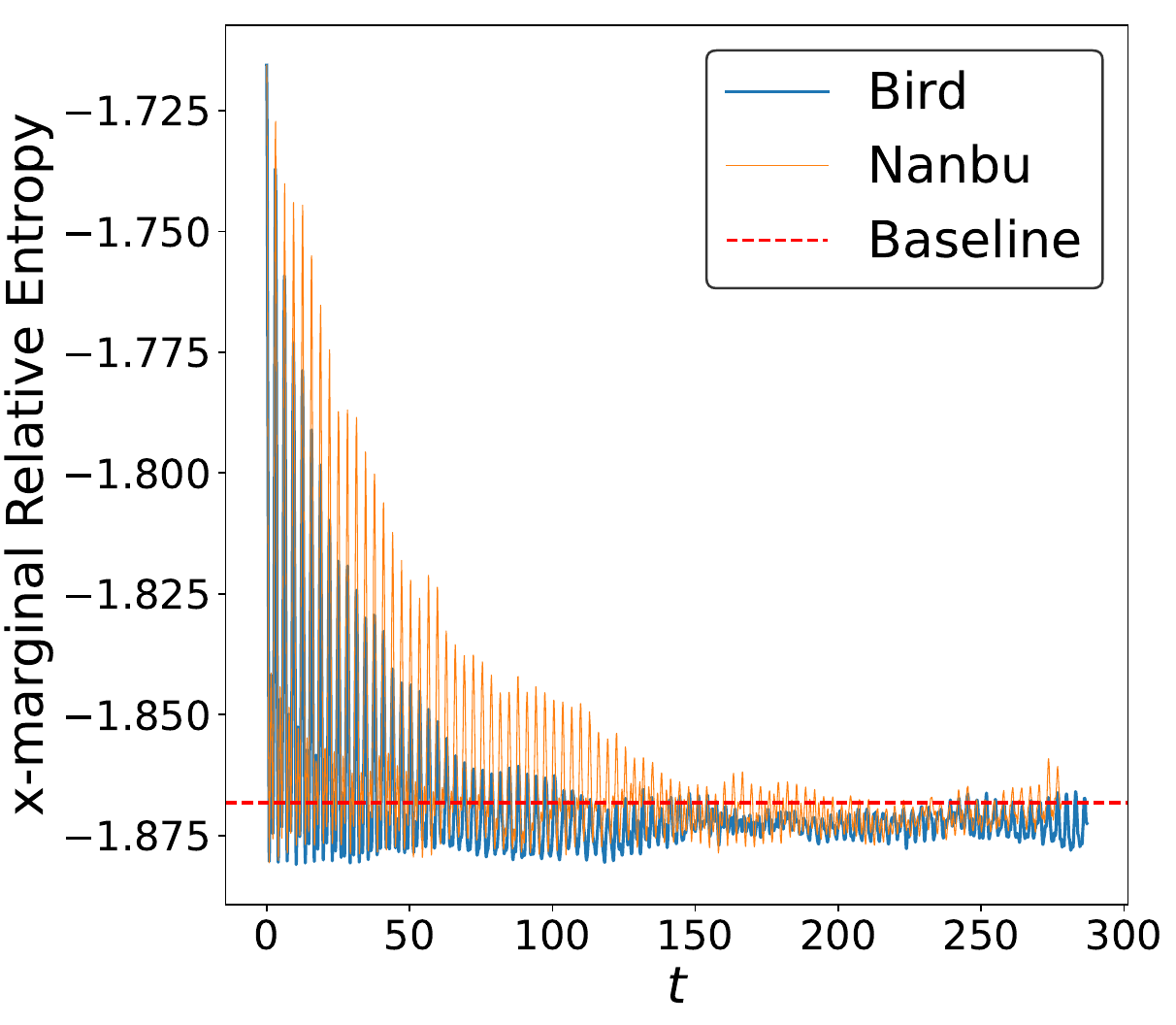}
  \includegraphics[width=0.22\textwidth]{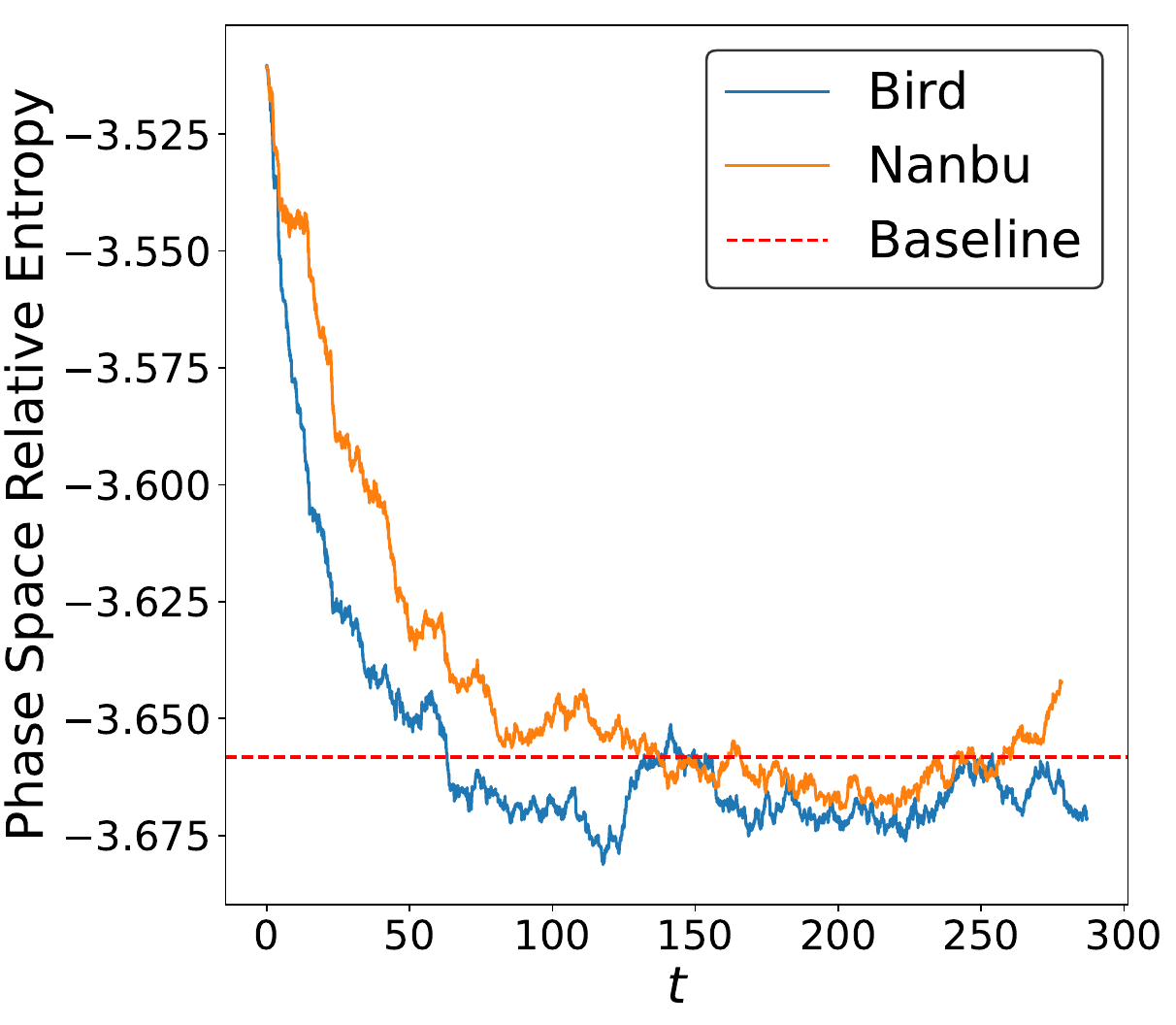}
  \caption{\textbf{Left Two Plots:} The contour of the target distribution $\rho^\ast \propto \exp(-0.5(x^2+y^2))$ and the samples obtained from the Nanbu sampler (Leftmost) and the Bird sampler (Second-Left). \textbf{Right Two Plots:}  The regularized relative entropy $KL_X^\delta$ (Second-Right) and $KL^\delta$ (Rightmost) of the Bird sampler and the Nanbu sampler. The baseline value is computed by using $N=1000$ samples obtained by the Inverse Transform method.
  }
  \label{fig:boltzmann_ex3}
\end{figure}

\begin{figure}[htbp]
  \centering
  \includegraphics[width=0.22\textwidth]{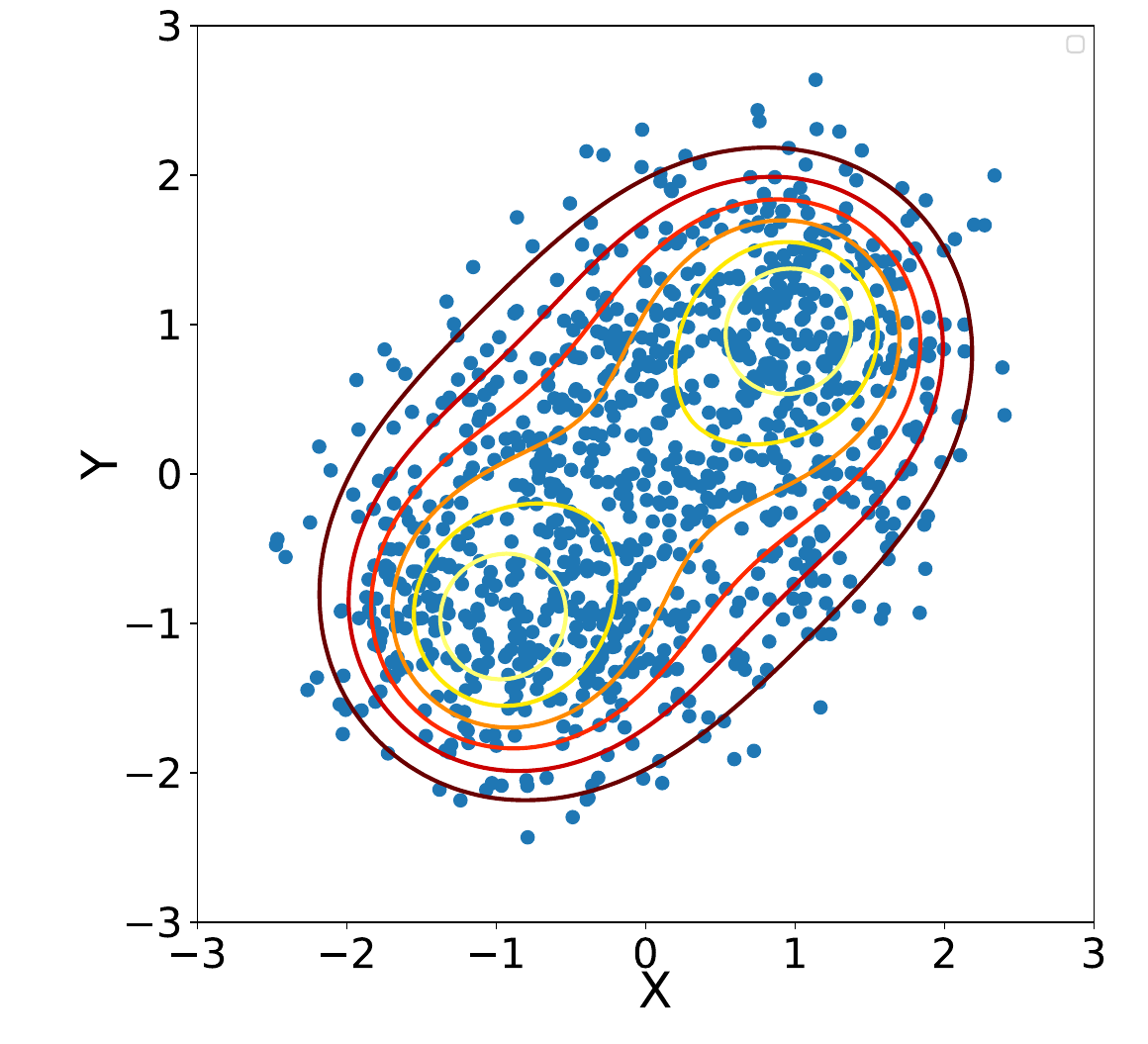}
  \includegraphics[width=0.22\textwidth]{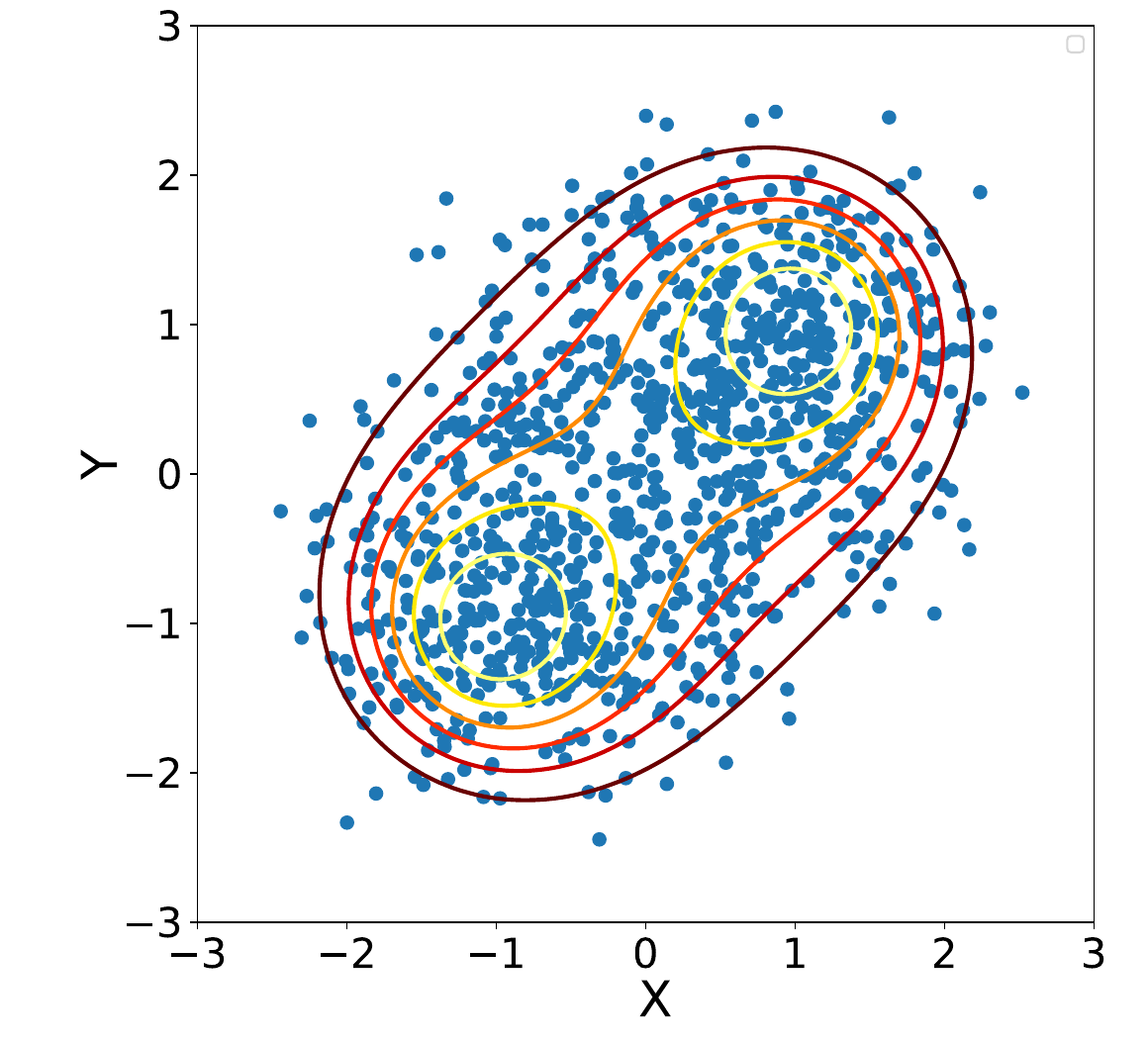}
  \includegraphics[width=0.22\textwidth]{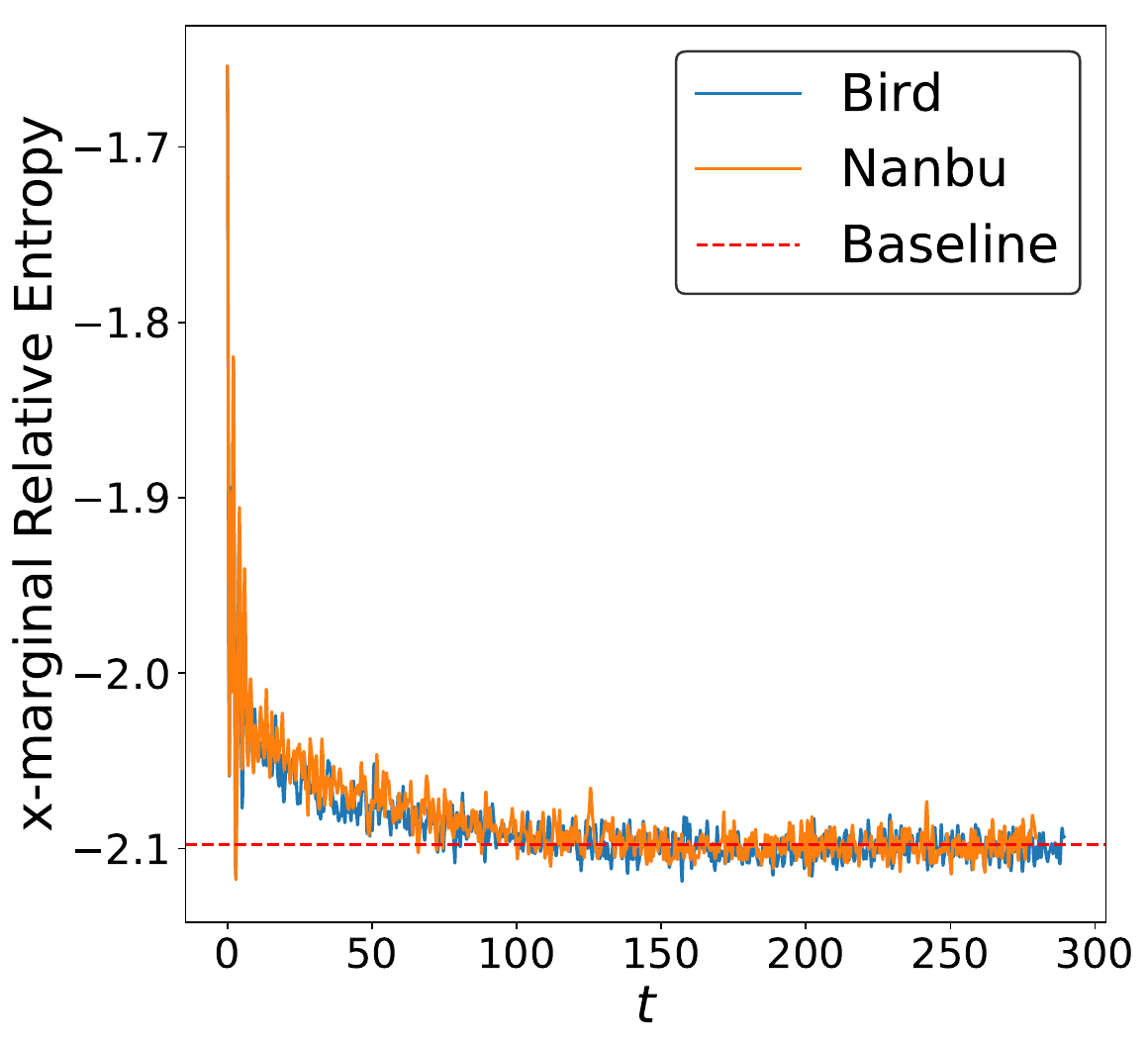}
  \includegraphics[width=0.22\textwidth]{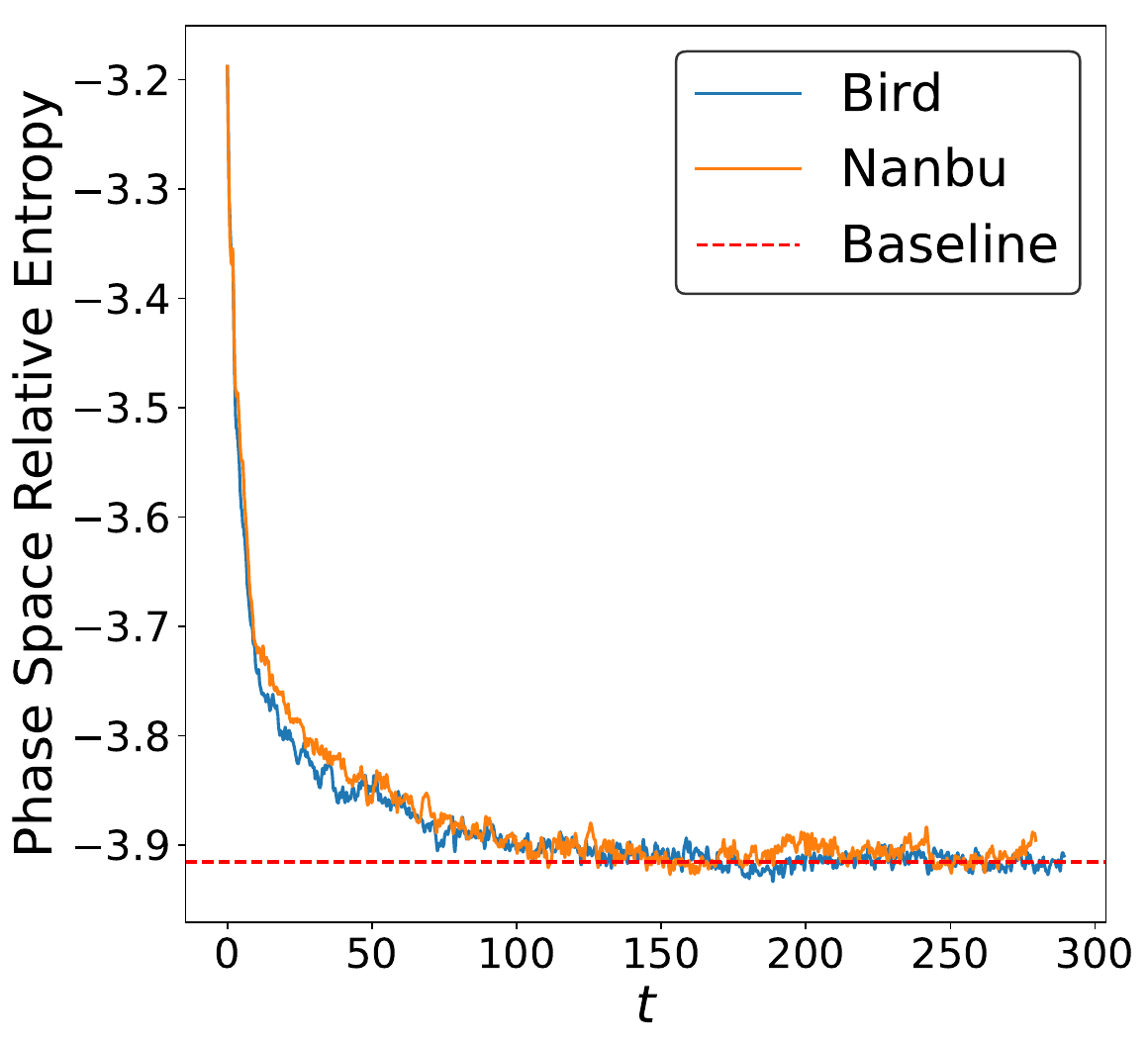}
  \caption{\textbf{Left Two Plots:} The contour of the target distribution $\rho^\ast \propto \exp(-0.1((x-1)^2+(y-1)^2)((x+1)^2+(y+1)^2))$ and the samples obtained from the Nanbu sampler (Leftmost) and the Bird sampler (Second-Left). \textbf{Right Two Plots:}  The regularized relative entropy $KL_X^\delta$ (Second-Right) and $KL^\delta$ (Rightmost) of the Bird sampler and the Nanbu sampler. The baseline value is computed by using $N=1000$ samples obtained by the Inverse Transform method.
  }
  \label{fig:boltzmann_ex4}
\end{figure}

\bibliographystyle{abbrv}
\bibliography{ref}

\end{document}